\documentclass[11pt]{amsart}
\usepackage[top=1in, bottom=1in, left=1in, right=1in]{geometry}
\usepackage{mathrsfs}
\usepackage{amsfonts}
\usepackage{amsmath}
\usepackage{amssymb}
\usepackage{bbm}
\usepackage{physics}
\numberwithin{equation}{section}
\usepackage{times}
\usepackage[usenames,dvipsnames]{color}
\usepackage{comment}
\usepackage{mathtools}
\usepackage{bbm}
\usepackage{bm}
\usepackage{esvect}
\usepackage{cite}
\usepackage{hyperref}
\hypersetup{
colorlinks = true,
linkcolor={black},
urlcolor={blue},
citecolor={blue},    
urlcolor = {blue},
citebordercolor = {0.33 .58 0.33},
 linkbordercolor = {0.99 .28 0.23},
 breaklinks=true}

\newcommand{\kommentar}[1]{}

\newcommand{\Z}{\mathbb{Z}}

\newcommand{\m}{\mathrm{m}}

\newcommand{\normone}[1]{\left\lVert#1\right\rVert}
\newtheorem{thm}{Theorem}[section]

\newtheorem{prop}[thm]{Proposition}

\newtheorem{lem}[thm]{Lemma}

\newcommand{\e}{\mathrm{e}}

\theoremstyle{remark}
\newtheorem{rem}{Remark}[section]

\usepackage{geometry}
\title{Almost all primes are not needed in Ternary Goldbach}

\author{Debmalya Basak, Raghavendra N. Bhat, Anji Dong and Alexandru Zaharescu}
\address{
Debmalya Basak: Department of Mathematics,
University of Illinois Urbana-Champaign,
Altgeld Hall, 1409 W. Green Street,
Urbana, IL, 61801, USA}
\email{dbasak2@illinois.edu}

\address{
Raghavendra N. Bhat: Department of Mathematics,
University of Illinois Urbana-Champaign,
Altgeld Hall, 1409 W. Green Street,
Urbana, IL, 61801, USA}
\email{rnbhat2@illinois.edu}

\address{
Anji Dong: Department of Mathematics,
University of Illinois Urbana-Champaign,
Altgeld Hall, 1409 W. Green Street,
Urbana, IL, 61801, USA}
\email{anjid2@illinois.edu}

\address{
Alexandru Zaharescu: Department of Mathematics,
University of Illinois Urbana-Champaign,
Altgeld Hall, 1409 W. Green Street,
Urbana, IL, 61801, USA and Simion Stoilow Institute of Mathematics of the Romanian Academy, 
P. O. Box 1-764, RO-014700 Bucharest, Romania}
\email{zaharesc@illinois.edu} 

\begin{document}
\setcounter{tocdepth}{1}
\begin{abstract}The ternary Goldbach conjecture states that every odd number $m \geqslant 7$ can be written as the sum of three primes. We construct a set of primes $\mathbb{P}$ defined by an expanding system of admissible congruences such that almost all primes are not in $\mathbb{P}$ and still, the ternary Goldbach conjecture holds true with primes restricted to $\mathbb{P}$.
\end{abstract}
\maketitle
\tableofcontents
\section{Introduction}\label{sec: Introduction}
In a letter to Euler in 1742, Goldbach stated that every even integer $m \geqslant 4$ can be written as the sum of two primes. This conjecture, known as Goldbach's conjecture, has been the subject of extensive research over the years. Although Goldbach's conjecture is still unsolved, many striking results have been obtained in several different directions. One such direction consists of results which show that almost all even positive integers can be written as a sum of two primes. The first non-trivial, although conditional, result in this direction was obtained by Hardy and Littlewood\cite{hardy1924}. Concerning the exceptional set (that is, the set of possible even integers that cannot be written as a sum of two primes), denoted by $\mathcal{E}$, Hardy and Littlewood\cite{hardy1924} showed that assuming the Generalized Riemann Hypothesis, the estimate
\[ E(X)\ll_{\varepsilon} X^{1/2+\varepsilon}
\]
holds for any $\varepsilon>0$ where $E(X)=\#\{n \leqslant X : n \in \mathcal{E}\}$. Since then, many results in this direction have been obtained over the years, see  Vinogradov \cite{vinogradov1937}, Van der Corput\cite{corput1937}, Cudakov\cite{cudakov1938},  Estermann\cite{estermann1938}, Vaughan\cite{vaughan1972},  Montgomery and Vaughan\cite{montgomery1975},  Chen and Liu\cite{chen1989, chen_2_1989}, Goldston \cite{goldston1992},  Li\cite{li2000,li_2_2000},  and Lu\cite{lu2010}. Currently the best known unconditional bound, proved in 2018, obtained by Pintz\cite{pintz2018}, states that there is a constant $X_2$ such that for any $X>X_2$, \[
E(X)<X^{0.72},
\] where $E(X)=\#\{n \leqslant X : n \in \mathcal{E}\}$.
\par
Another direction consists of results that show that all large even numbers can be represented as a sum of a prime and a number that is a product of bounded number of prime factors. Results in this direction were established by Rényi\cite{renyi1948}, Pan\cite{pan1962}, and Wang\cite{wang1962}, culminating with Chen's theorem\cite{chen1966, chen1973}, which shows that every large even number can be written as a sum of a prime and a product of at most two prime factors. Generalizations of Chen’s theorem have also been obtained. Lu and Cai \cite{lucai1998, lucai1999} proved more general versions of Chen's theorem for short intervals and arithmetic progressions. Cai\cite{cai2002} further proved that every large enough even integer $n$ is a sum of a prime less than $n^{0.95}$ and a number with at most two prime factors. An explicit version of Chen's theorem was proved by Bordignon, Johnston and Starichkova \cite{bordignon2022}.
\par
In yet another direction, results have been obtained showing that all large enough even numbers can be written as a sum of two primes plus a bounded number of powers of 2. The reader is referred to Linnik\cite{Linnik1951}, Gallagher\cite{gallagher1975}, Li\cite{li2000_powers}, Heath-Brown and Putcha\cite{HeathBrown2002}, Pintz and Ruzsa\cite{Pintz2003}, as well as Khalfalah and Pintz\cite{Khalfalah2006}. 
\par
In his correspondence with Euler, Goldbach also proposed that every odd integer $m \geqslant 7$ can be written as the sum of three primes. This is known as the ternary Goldbach conjecture (also called Goldbach's weak conjecture). Hardy and Littlewood showed \cite{hardy1922} assuming the generalized
Riemann hypothesis, that every sufficiently large odd number is a
sum of three primes. Vinogradov \cite{vinogradov1937} used his new estimates on exponential sums over primes to prove unconditionally that the ternary Goldbach conjecture is true for all sufficiently large odd numbers. More recently, Helfgott\cite{helfgott2013} proved the conjecture for all odd numbers larger than or equal to 7.
\par
In the present paper, our primary aim is to construct a set of primes $\mathbb{P}$ such that almost all primes are not in $\mathbb{P}$ and still, the ternary Goldbach conjecture holds true with primes taken from $\mathbb{P}$. We will prove the following result.
\begin{thm}\label{Main Theorem}
There exists a set of primes $\mathbb{P}$ such that
\begin{itemize}
\item[(a)] For any positive integer $B$, 
\[
\lim_{X \to \infty} \frac{(\log X)^B\cdot\# \{ p \in \mathbb{P} : p \leqslant X \} }{X}=0, \quad \textnormal{and} 
\]
\item[(b)] Every odd number $m \geqslant 7$ is the sum of three primes in $\mathbb{P}$.
\end{itemize}
\end{thm}
\par
The proof of Theorem \ref{Main Theorem} relies on a concrete construction of a set of primes $\mathbb{P}$ satisfying a growing class of congruence constraints as the size of the primes increase. This is responsible for our set $\mathbb{P}$ to satisfy condition (a) in the statement of Theorem \ref{Main Theorem}. As we shall see later, the above class of congruence constraints are governed by products of small primes.  At each stage of our construction, we employ the Hardy--Littlewood Circle method in order to establish asymptotic formulas for the number of representations as sums of three primes from $\mathbb{P}$. Collectively, all the stages together imply condition (b) in the statement of Theorem \ref{Main Theorem}.
\subsection*{Structure of the Paper} The paper is organized as follows. We introduce some standard notations in Section \ref{sec: general notations}. Section \ref{sec: Overall Strategy} outlines our approach to proving Theorem \ref{Main Theorem}, showing that the proof hinges on a crucial lemma. Section \ref{sec: Initial Setup} addresses the construction of our specific set of primes and demonstrates that this construction fulfills condition (a) of our key lemma. The setup for applying the Hardy--Littlewood method, which is used to address condition (b) of the lemma, is also introduced at the end of Section \ref{sec: Initial Setup}. Sections \ref{Minor Arcs : Sec} and \ref{Major Arcs : Sec} provides the details of the minor arc and major arc estimates respectively. The delicate decomposition of the arithmetic factor is presented across Sections \ref{sec: The arithmetic factor} and \ref{sec: Arithmetic factor II}. Finally, in Section \ref{sec: Completion}, we complete the proof of part (b) of the equivalent lemma, thereby concluding the proof of our main theorem.
\section{General Notations}\label{sec: general notations} We employ some standard notation that will be used throughout the article.
\begin{itemize} 
\item The expressions $f(X)=O(g(X))$, $f(X) \ll g(X)$, and $g(X) \gg f(X)$ are equivalent to the statement that $|f(X)| \leqslant C|g(X)|$ for all sufficiently large $X$, where $C>0$ is an absolute constant. A subscript of the form $\ll_{\alpha}$ means the implied constant may depend on the parameter $\alpha$. Dependence on several parameters is indicated similarly, as in $\ll_{\alpha, \lambda}$.
\item The notation $a \asymp A$ means that there exist absolute constants $c_1,c_2>0$ such that $c_1A \leqslant a \leqslant c_2A$.
\item The function $\varphi$ denotes the Euler totient function.
\item The function $\Lambda$ denotes the Von-Mangoldt function.
\item The function $\omega(n)$ denotes the prime omega function, which counts the number of prime divisors of a natural number $n$. 
\item The M\"{o}bius function denoted by $\mu$, is a multiplicative function defined in the following way: for each prime $p$,
\[\mu(p^k) := \begin{cases}
  -1,  & k=1 \\
  0, & k \geqslant 2.
\end{cases}\]
    \item Given two arithmetic functions $f$ and $g$, $f*g$ denotes the Dirichlet convolution of $f$ and $g$.
\item The notation $e(x)$ stands for $\exp(2\pi i x)$.
\item Given $\alpha \in \mathbb{R}$, the notation $\|\alpha\|$ denotes the smallest distance of $\alpha$ to an integer.
\item Given a prime $p$, the notation $\Z_{p}$ refers to the set of residue classes modulo $p$.
\item The notation $\#(S)$ stands for the cardinality of a set $S$. 
\item The notation $p^{v}||q$ means $p^{v}$ is the exact power of $p$ that divides $q$. 
\end{itemize}
\section{Proof of Main Theorem: Overall Strategy}\label{sec: Overall Strategy}
In order to prove Theorem \ref{Main Theorem}, we will prove the following lemma.

\begin{lem}\label{Main Lemma}
Fix a positive integer $A$. There exists $u_A>0$ such that for all $u\geqslant u_A$,
\begin{itemize}
\item[(a)] There exists a set $\mathbb{P}_u$ of primes in $[u, 2u]$ satisfying
\begin{align}\label{Main Lemma Inequality}
\# (\mathbb{P}_u) \leqslant \frac{u}{(\log u)^A}, \quad \textnormal{and} 
\end{align}
\item[(b)] For any odd number $m \in [4u,5u]$, there exist $p_1,p_2, p_3 \in \mathbb{P}_{u}$ such that $m=p_1+p_2+p_3$.
\end{itemize}
\end{lem}
Assuming Lemma \ref{Main Lemma}, let us first prove Theorem \ref{Main Theorem}.
\par
\begin{proof}[Proof of Theorem \ref{Main Theorem}]
We begin with the following construction. Fix absolute constants $c_1,c_2,c_3,c_4$ such that
\begin{align}\label{Choice of Constants}
c_1<1<c_2<\frac{5c_1}{4} <c_3< \frac{5c_2}{4}<c_4<\frac{5c_3}{4}<2<\frac{5c_4}{4}.
\end{align}
For each $A\in \mathbb{N}$, Lemma \ref{Main Lemma} provides us with $u_A$. Define $\ell_A$ be the smallest integer such that 
\begin{align}\label{l_A definition}
u_A \leqslant c_1 2^{\ell_A} .
\end{align}
For each $\ell_A \leqslant \ell< \ell_{A+1}$, we apply Lemma \ref{Main Lemma} repeatedly with the choices $u_i= c_i 2^{\ell}$, where $1 \leqslant i \leqslant 4$. By Lemma \ref{Main Lemma}, we obtain the sets of primes $\mathbb{P}_{A,u_i} \subseteq [c_i 2^{\ell}, 2 c_i 2^{\ell}]$. Let $\mathbb{P}_{A,\ell} = \bigcup_{i=1}^{4} \mathbb{P}_{u_i} $. Then for each $\ell_A \leqslant \ell< \ell_{A+1}$ we have:
\begin{itemize}
\item[(i)] $\mathbb{P}_{A,\ell} \subseteq [c_1 2^{\ell}, 2 c_4 2^{\ell}]$ is a set of primes satisfying
\[
\# (\mathbb{P}_{A,\ell}) \leqslant \frac{4c_42^{\ell}}{(\log c_12^{\ell})^A}.
\]
    
\item[(ii)] For any odd number $m \in [2^{\ell+2}, 2^{\ell+3}]$, there exist $p_1,p_2, p_3 \in \mathbb{P}_{A,\ell}$ such that $m=p_1+p_2+p_3$.
\end{itemize}
To see this, observe that the first condition is satisfied by maximizing the right hand side of \eqref{Main Lemma Inequality} as $u$ varies over the $u_i$'s. For the second condition, note that for any $m\in [2^{\ell+2}, 2^{\ell+3}]$, $m \in [4u_i,5u_i]$ for some $1 \leqslant i \leqslant 4$.
\par
Let us define
\begin{align}\label{Prime Sets Definitions}
\mathbb{P}_{0} =\{ p \textrm{ prime} : p \leqslant 20u_1 \}, \quad \mathbb{P}_1 = \bigcup_{A \geqslant 1} \bigcup_{\ell_A \leqslant \ell < \ell_{A+1}} \mathbb{P}_{A,\ell} \quad \textrm{and} \quad \mathbb{P} = \mathbb{P}_{0} \cup \mathbb{P}_1.
\end{align}
Condition (b) from the statement of Theorem \ref{Main Theorem} holds trivially for all odd $m$ up to $20u_1$ due to Helfgott \cite{helfgott2013}. So we may assume $m >20u_1$. Suppose $m \in [2^{\ell+2},2^{\ell+3})$ for some $\ell \in \mathbb{N}$. Since $m>20u_1$, we must have $\ell \geqslant \ell_1$. Suppose $\ell_A \leqslant \ell< \ell_{A+1}$ for some $A \in \mathbb{N}$.  By our construction of the set $\mathbb{P}_{A,\ell}$, there exist $p_1,p_2,p_3 \in \mathbb{P}_{A,\ell} \subseteq \mathbb{P}$ such that $m=p_1+p_2+p_3$. Therefore condition (b) is satisfied.
\par
Now we prove condition (a). Fix $X \geqslant 2$ large. Let $A_0$ to be the smallest positive integer such that
\begin{align}\label{Define A_0}
X <c_1 2^{\ell_{A_0}}.
\end{align}
Let $B_0$ be the smallest positive integer such that
\begin{align}\label{Define B0}
   \sqrt{X}< c_1 2^{\ell_{B_0}}.
\end{align} 
Then we have
\begin{align}
\# \{ p \in \mathbb{P} : p \leqslant X \} &\leqslant \sum_{1 \leqslant A<A_0}\sum_{\ell_A \leqslant \ell<\ell_{A+1}}  \# \{ \mathbb{P}_{A,\ell}\} \notag \\
& \leqslant \sum_{1 \leqslant A<B_0} \sum_{\ell_A \leqslant \ell<\ell_{A+1}}   \# \{ \mathbb{P}_{A,\ell} \}+\sum_{B_0 \leqslant A<A_0} \sum_{\ell_A \leqslant \ell<\ell_{A+1}}   \# \{ \mathbb{P}_{A,\ell} \} \notag \\
& \ll_{B_0} \sqrt{X}\log X + \frac{X}{(\log X)^{B_0-1}} \ll_{B_0} \frac{X}{(\log X)^{B_0-1}}. \label{Primes Bound involving B0}
\end{align}
Note that $u_A$ and $\ell_A$ are monotonically increasing as $A$ increases. Hence for any fixed positive integer $B$, letting $X$ tend to infinity, we can choose $B_0=B+1$ in \eqref{Define B0} and \eqref{Primes Bound involving B0}, thereby satisfying condition (a).
\end{proof}
The proof of Lemma \ref{Main Lemma} is more intricate and will be presented across the remainder of the paper.
\section{Initial Setup}\label{sec: Initial Setup}
In this section, we lay some preliminary groundwork for the proof of Lemma \ref{Main Lemma}.
\subsection{Residue Class Constructions}
Let $A \in \mathbb{N}$ be fixed and suppose $u$ be sufficiently large depending on $A$. Let 
\begin{align}\label{define Z}
z=3A\log\log u.
\end{align}
Enumerate the primes $p_1<p_2<\dots< p_k \leqslant z$ and let 
\begin{align}\label{Define q_0}
q_0= p_1p_2\cdots p_k.  
\end{align}
Then we have
\begin{align}\label{Size of q0}
q_0=\prod _{p<z} p \asymp   e^z =e^{3A\log\log u} =(\log u)^{3A}.
\end{align}
For each $j\in \{1,\dots, k\}$, let 
\begin{align}\label{Aj Definition}
a_j =\lceil p_j^{1/3}\rceil +1
\end{align}
and define 
\begin{align}\label{Rj Definition}
\mathcal{R}_j=\mathcal{R}_{j,1} \cup \mathcal{R}_{j,2} \cup \mathcal{R}_{j,3},
\end{align}
where 
\begin{align}\label{R123 Definition}
\mathcal{R}_{j,1} =\{1,2,\dots ,a_j\},\quad \mathcal{R}_{j,2}=\{a_j,2a_j,\dots ,a_j^2\} \quad
\textrm{and} \quad \mathcal{R}_{j,3} &=\{a_j^2,2a_j^2,\dots ,a_j^3\}.
\end{align}
We start with the following lemma.
\begin{lem}\label{sum of three numbers modulo pj}
For any $n\in \mathbb{N} $ and $1 \leqslant j \leqslant k$, there exist $r_1,r_2,r_3\in \mathcal{R}_{j}$ such that $ n\equiv r_1+r_2+r_3 \mod p_j.$
\end{lem}
\begin{proof}
Since we are looking for solutions modulo $p_j$, it suffices to consider $n\in \Z/p_j\Z$, i.e. $n=1,2,\cdots, p_j$. For any $n$, since $a_j>p_j^{1/3}, $we can write $n$ in base $a_j$, that is, \[n=w_2a_j^2+w_1a_j+w_0,\] where $w_i<a_j$ for $i=1,2,3$. Note that $a_j\geqslant 3$. We now discuss in cases.\\
    
\noindent \textbf{Case 1} : Suppose $w_i\neq 0$ for all $w_i$. In this case, choosing $r_1=w_0, r_2 = w_1a_j, $ and $r_3=w_2a_j^2$ will yield the desired result.\\

\noindent \textbf{Case 2} : Suppose exactly one of the $w_i$ is 0. Assume that $w_0=0$. Then, if either $w_1$ or $w_2$ is greater than 1, say $w_1\geqslant 2$, then we have $n = w_2a_j^2+(w_1-1)a_j+a_j$. Otherwise we write $n = a_j^2+a_j = (a_j-1)a_j + a_j + a_j$. Now assume $w_1=0$. Again, if either $w_0$ or $w_2$ is greater than 1, then we are done. Otherwise, $n = a_j^2+1 = (a_j-1)a_j+a_j+1$. Finally, assume $w_2=0$. If either $w_0$ or $w_1$ is greater than 1, then split same as in other sub-cases. Otherwise, $n=a_j+1$, then $n=(a_j-1)+1+1$. \\

\noindent \textbf{Case 3} : Suppose exactly two of the $w_i$ are 0. Assume $n=w_2a_j^2$. If $w_2\geqslant 3$, then we have $n=a_j^2+a_j^2+(w_2-2)a_j^2$. If $w_2 = 2$, then $n = a_j^2 + (a_j-1)a_j + a_j$. If $w_2 = 1$, and if $p_j>8$, then $n = a_j +a_j+(a_j-2)a_j$; otherwise, we write $n=a_j+1+(a_j-1)$. Now, assume $n=w_1a_j$. If $w_1\geqslant 3$, then split in the same way as the above sub-case. If $w_1=2$, then $n=a_j + (a_j-1)+1$. If $w_1=1$, then $n=1+1+(a_j-2)$. Finally, consider the case when $n=w_0$. If $w_0\geqslant 3$, then $n = 1+1+(w_0-2)$. Otherwise, we will write $w_0 + p_j$ instead as the triple sum. Note that $w_0 + p_j$ is at most $p_j+2$. Since $a_j^3 \geqslant (p_j^{1/3}+1)^3 = p_j+7$, we can always write $w_0+p_j = w'_2a_j^2+w'_1a_j+w'_0$, where $w'_2>0$, and so will be reduced to previous cases.
\end{proof}
Note that
\begin{align}\label{Rj Set Size}
\#(\mathcal{R}_j)=3a_j= 3(\rceil p_j^{\frac{1}{3}}\rceil +1)=  3p_j^{\frac{1}{3}}+O(1).
\end{align}
Let   ${\mathcal{R}}_0$ be the set of integers $r\in[1,q_0] $  such that $(r,q_0) =1$ and for each $1 \leqslant j \leqslant k$, 
\[ r \equiv r_j \bmod p_j,
\]
for some $r_j \in \mathcal{R}_{j}$. By the Chinese Remainder Theorem, for $u$ large with respect to $A$, we have
\begin{align}\label{R0 Set Size}
\#(\mathcal{R}_0)= \prod_{j=1} ^k \#(\mathcal{R}_j) &= \prod_{j=1} ^k (3p_j^{\frac{1}{3}}+O(1)) \asymp q_0^{1/3}(\log u)^{\varepsilon},
\end{align}
for any $\varepsilon>0$. Define $\mathbb{P}_u$ to be the set of  primes $p$ in $[u, 2u]$ such that $p\equiv r \bmod q_0$ for some $r\in \mathcal{R}_0$. Trivially, one has
\begin{align}\label{Prime Set Bound 2}
\# (\mathbb{P}_u) \leqslant \frac{u}{q_0}\# (\mathcal{R}_0)\ll \frac{u (\log u)^{\varepsilon}}{(\log u)^{2A}}
\end{align}
for any $\varepsilon>0$. Therefore, in order to prove Lemma \ref{Main Lemma}, it suffices to show that for any odd $m \in[4 u, 5 u]$ that there exist $p_1, p_2, p_3 \in \mathbb{P}_u$ such that
$$
m=p_1+p_2+p_3.
$$	
We will do so using the Hardy--Littlewood Circle method.
\subsection{Major and Minor Arc Definitions} Let
\begin{align}\label{Sum Definition}
S(\alpha)=\sum_{\substack{u \leqslant k \leqslant 2u\\ k \bmod q_0 \in \mathcal{R}_0}}\Lambda(k) e(k \alpha)
\end{align}
where $\Lambda(n)$ denotes the Von-Mangoldt function. Let 
\begin{align}\label{Representation of m}
\mathfrak{R}(m) = \underset{\substack{u \leqslant k_1,k_2,k_3 \leqslant 2u \\ k_i \bmod q_0 \in \mathcal{R}_0 \\ k_1+k_2+k_3 = m}}{\sum \sum \sum} \Lambda(k_1)\Lambda(k_2) \Lambda(k_3).
\end{align} 
Our goal is to provide an asymptotic formula for $\mathfrak{R}(m)$ which we achieve in \eqref{Asymptotic Formula} below. To proceed, we begin by noting that the orthogonality relation
$$
\int_0^1 e(\alpha h) \, \dd  \alpha= \begin{cases}1 & \text { when } h=0, \\ 0 & \text { when } h \neq 0,\end{cases}
$$
gives
\begin{align}
\mathfrak{R}(m)=\int_0^1 S(\alpha)^3 e(-m \alpha)\, \dd \alpha. \label{Representation of m integral}
\end{align}
Therefore to prove condition (b) of Lemma \ref{Main Lemma}, we will estimate the integral on the right hand side of \eqref{Representation of m integral} To do so, we first define the major and minor arcs. Let $P = (\log u)^{B}$, $Q = u(\log u)^{-B}$ where $B>0$ will be chosen later in terms of $A$. When $1 \leqslant a \leqslant q \leqslant P$ and $(a, q)=1$, let
\begin{align}\label{Major Arcs}
\mathfrak{M}(q, a)=\left\{\alpha:|\alpha-a / q| \leqslant 1/Q\right\} .
\end{align}
We call $\mathfrak{M}(q, a)$ the major arcs. Let $\mathfrak{M}$ denote the union of the $\mathfrak{M}(q, a)$'s. The set $\mathfrak{m}=[0,1] \backslash \mathfrak{M}$ forms the minor arcs.
When $a / q \neq a^{\prime} / q^{\prime}$ and $q, q^{\prime} \leqslant P$, one has
$$
\left|\frac{a}{q}-\frac{a^{\prime}}{q^{\prime}}\right| \geqslant \frac{1}{q q^{\prime}}\geqslant \frac{1}{P^2} >\frac{2}{Q}.
$$
Thus the $\mathfrak{M}(q, a)$ are pairwise disjoint. By \eqref{Representation of m integral}, we have
\begin{align}\label{Major+Minor Arcs}
\mathfrak{R}(m) = \int_{\mathfrak{M}}S(\alpha)^3 e(-\alpha m) \, \dd   \alpha+\int_{\mathfrak{m}} S(\alpha)^3 e(-\alpha m) \, \dd   \alpha.
\end{align}
We treat the integral over the minor arcs, major arcs and the associated arithmetic factor across Sections \ref{Minor Arcs : Sec}, \ref{Major Arcs : Sec}, \ref{sec: The arithmetic factor} and \ref{sec: Arithmetic factor II}.
\section{Minor Arcs} \label{Minor Arcs : Sec}
Our goal in this section is to estimate the sum
\begin{align}\label{Minor Arc Sum}
S(\alpha)=\sum_{r \in \mathcal{R}_0} \sum_{\substack{u \leqslant k \leqslant 2 u \\ k \equiv r \bmod q_0}} \Lambda(k) e(k \alpha).
\end{align}
when $\alpha \in \mathfrak{m}$. We first record the following estimate for exponential sums over primes.
\begin{lem}\label{Vinogradov bound}
Let $u\geqslant 1$, $\alpha\in \mathbb{R}$ and consider a reduced fraction $a/q$ such that $|\alpha-a/q|\leqslant 1/q^2$. Then
\begin{align}
\sum_{1 \leqslant k \leqslant u} \Lambda(k)e(k\alpha) \ll \bigg(\frac{u}{q^{1/2}} +u^{4/5} +u^{1/2}q^{1/2}\bigg)(\log u)^4.
\end{align}
\end{lem}
\begin{proof}
See \cite[Ch. 25]{D2000}.
\end{proof}
Using the orthogonality relation
\[
\mathbbm{1}_{k \equiv r \bmod q_0} = \frac{1}{q_0} \sum_{\ell=0}^{q_0-1} e \bigg( \frac{ \ell(k-r)}{q_0} \bigg ),
\]
we write
\begin{align}\label{Detecting AP}
S(\alpha) = \frac{1}{q_0} \sum_{r \in \mathcal{R}_0}  \sum_{\ell=0}^{q_0-1} e \bigg( \frac{- \ell r}{q_0} \bigg ) \sum_{\substack{u \leqslant k \leqslant 2 u}} \Lambda(k) e \bigg (k \bigg(\alpha +\frac{\ell}{q_0} \bigg)\bigg).
\end{align}
To treat the innermost sum in \eqref{Detecting AP}, we will derive a variant of Lemma \ref{Vinogradov bound} to address the special case when $\alpha\in \mathbb{R}$ and there exists some reduced fraction $a/q$ such that $|\alpha-a/q|\leqslant\gamma/q^2$, for some $\gamma \geqslant 1$.
\begin{lem}\label{generalized vinogradov bound}
    Let $u\geqslant 1$, $\alpha\in \mathbb{R}$ and consider a reduced fraction $a/q$ such that $|\alpha-a/q|\leqslant\gamma/q^2$ for some $\gamma\geqslant 1$. Then
\begin{align} \label{Vinogradov with gamma}
\sum_{1 \leqslant k \leqslant u} \Lambda(k)e(k\alpha) \ll  \bigg(\frac{\gamma^{1/2}u}{q^{1/2}} +u^{4/5} +u^{1/2}q^{1/2}\bigg)(\log u)^4.
\end{align}
\end{lem}
\begin{proof}
Dirichlet's theorem on Diophantine approximation (see \cite[Lemma 2.1]{V1997}) asserts that for any real $\alpha$ and for any real number $M\geqslant 1$, there exists a rational number $a_1/q_1$, where $1 \leqslant q_1 \leqslant M$, and $(a_1,q_1)=1$, such that $\lvert \alpha-a_1/q_1 \rvert \leqslant 1/(q_1M)\leqslant1/q_1^2$. Choose $\alpha$ to be as in the statement of the lemma and $M=2q$. Then $1\leqslant q_1\leqslant 2q$, $(a_1,q_1)=1$, and
\begin{align}
    \left\lvert \alpha-\frac{a_1}{q_1} \right\rvert \leqslant \frac{1}{2qq_1}\leqslant\frac{1}{q_1^2}.\label{Dirichlet theorem}
\end{align}
By Lemma \ref{Vinogradov bound}, we have 
\begin{align}
\sum_{1 \leqslant k \leqslant u} \Lambda(k)e(k\alpha) \ll \bigg(\frac{u}{q_1^{1/2}} +u^{4/5} +u^{1/2}q_1^{1/2}\bigg)(\log u)^4.\label{vinogradov bound}
\end{align}
In the special case when $a/q=a_1/q_1$, we must have $a_1=a$ and $q_1=q$. Therefore \eqref{vinogradov bound} implies that 
     \begin{align}
\sum_{1 \leqslant k \leqslant u} \Lambda(k)e(k\alpha) \ll \bigg(\frac{u}{q^{1/2}} +u^{4/5} +u^{1/2}q^{1/2}\bigg)(\log u)^4.\label{case 1 general bound}
\end{align}
Otherwise assume that $a/q\neq a_1/q_1$. Then using \eqref{Dirichlet theorem}, we have
\begin{align}
    \frac{1}{qq_1}&\leqslant\left\lvert\frac{a_1}{q_1}-\frac{a}{q}\right\rvert\leqslant\left\lvert\frac{a_1}{q_1}-\alpha\right\rvert+\left\lvert\alpha-\frac{a}{q}\right\rvert\leqslant \frac{1}{2qq_1}+\frac{\gamma}{q^2}.\label{inequality}
\end{align}
This implies that $1/(2qq_1)\leqslant\gamma/q^2$, which is equivalent to 
\begin{align}
\frac{1}{q_1^{1/2}}\leqslant\frac{(2\gamma)^{1/2}}{q^{1/2}}.\label{inequality on q1}
\end{align}
Then using \eqref{vinogradov bound} and \eqref{inequality on q1}, we deduce that 
 \begin{align}
\sum_{1 \leqslant k \leqslant u} \Lambda(k)e(k\alpha) \ll \bigg(\frac{\gamma^{1/2} u}{q^{1/2}} +u^{4/5} +u^{1/2}q^{1/2}\bigg)(\log u)^4.\label{case 2 general bound}
\end{align}   
Therefore combining \eqref{case 1 general bound} and \eqref{case 2 general bound}, we get our desired result.
\end{proof}
Let's revert back to the sum in the right hand side of \eqref{Detecting AP} when $\alpha \in \mathfrak{m}$. By Dirichlet's theorem on Diophantine approximation (see \cite[Lemma 2.1]{V1997}), for any real $\alpha$ and for any real number $Q\geqslant 1$, there exists a rational number $a/q$ such that $\lvert \alpha-a/q \rvert \leqslant 1/(qQ), 1 \leqslant q \leqslant Q$, and $(a,q)=1$. If $q \leqslant P$, then $\alpha \in \mathfrak{M}(q,a)$; hence if $\alpha \in \mathfrak{m}$, then $P<q<Q$. Therefore given $\alpha \in \mathfrak{m}$, there exists a reduced fraction $a/q$ with $\lvert \alpha-a/q \rvert \leqslant 1/(qQ) \leqslant 1/q^2, (a,q)=1$ and $P<q \leqslant Q$. Fix $\ell \in \{1,2,\cdots, q_0-1\}$ and let
\[
\alpha_{\ell} = \alpha+\frac{\ell}{q_0}.
\]
We write
\[
\frac{a}{q}+\frac{\ell}{q_0} = \frac{aq_0+\ell q}{qq_0} = \frac{a_{\ell}}{q_{\ell}},
\]
where $(a_{\ell},q_{\ell})=1$. There is some ambiguity in this notation because of the $a$-dependence in $q_{\ell}$. This needs to be remembered in the calculations. We have
\[
\bigg \lvert \alpha_{\ell} - \frac{a_{\ell}}{q_{\ell}} \bigg \rvert \leqslant \frac{1}{q^2} \leqslant \frac{q_0^2}{q_{\ell}^2}.
\]
Therefore applying Lemma \ref{generalized vinogradov bound} with $\alpha = \alpha_{\ell}$ and $\gamma = q_0^2$, we obtain
\begin{align} \label{Vinogradov with Tilde Alpha}
\sum_{\substack{u \leqslant k \leqslant 2 u}} \Lambda(k) e \bigg (k \left(\alpha +\frac{\ell}{q_0} \right)\bigg) &\ll  \bigg(\frac{q_0 u}{q^{1/2}} +u^{4/5} +u^{1/2}q^{1/2}\bigg)(\log u)^4.
\end{align}
Since $q >(\log u)^B$ and $q_0 \asymp (\log u)^{3A}$, we arrive at
\begin{align}\label{Minor Arcs Prefinal Step 1}
\sum_{\substack{u \leqslant k \leqslant 2 u}} \Lambda(k) e \bigg (k \left(\alpha +\frac{\ell}{q_0} \right)\bigg) \ll \frac{u}{(\log u)^{\frac{B}{2}-3A-4}}.
\end{align}
Substituting \eqref{Minor Arcs Prefinal Step 1} in \eqref{Detecting AP} and noting that $\#(\mathcal{R}_0) \asymp q_0^{1/3}(\log u)^{\varepsilon}$ for any $\varepsilon>0$, it follows that
\begin{align}\label{Minor Arcs Prefinal Step 2}
S(\alpha) \ll \frac{1}{q_0} \sum_{r \in \mathcal{R}_0}  \sum_{\ell=0}^{q_0-1} \frac{u}{(\log u)^{\frac{B}{2}-3A-4}} \ll_{\varepsilon} \frac{u}{(\log u)^{\frac{B}{2}-4A-4-\varepsilon}}.
\end{align}
Now we are ready to estimate the integral over the minor arcs. We write
\begin{align}
\bigg \lvert \int_{\mathfrak{m}} S(\alpha)^3 e(-\alpha m) \, \dd   \alpha \bigg \rvert &\leqslant \left ( \max_{\alpha \in \mathfrak{m}} S(\alpha) \right) \int_{\mathfrak{m}} \lvert S(\alpha)\rvert ^2  \, \dd   \alpha \notag \\
& \leqslant \left ( \max_{\alpha \in \mathfrak{m}} S(\alpha) \right) \int_{0}^{1} \lvert S(\alpha)\rvert ^2  \, \dd   \alpha. \label{Minor Arc Integral}
\end{align}
The integral on the right hand side of \eqref{Minor Arc Integral} is
\begin{align}
\sum_{r_1, r_2 \in \mathcal{R}_0} \sum_{\substack{u \leqslant k_1,k_2 \leqslant 2 u \\ k_1 \equiv r_1 \bmod q_0 \\ k_2 \equiv r_2 \bmod q_0}} \Lambda(k_1)  \Lambda(k_2) \int_{0}^{1} e((k_1-k_2)\alpha) \, \dd   \alpha &=  \sum_{r \in \mathcal{R}_0} \sum_{\substack{u \leqslant k \leqslant 2 u \\ k \equiv r \bmod q_0}} \Lambda(k) ^2 \notag\\
&\ll_{\varepsilon} u(\log u)^{A+1+\varepsilon}. \label{Conjugate Integral Bound}
\end{align}
Substituting \eqref{Minor Arcs Prefinal Step 2} and \eqref{Conjugate Integral Bound} in \eqref{Minor Arc Integral}, we obtain
\begin{align} \label{Minor Arcs Prefinal Step 3}
\int_{\mathfrak{m}} S(\alpha)^3 e(-\alpha m) \, \dd   \alpha \ll_{\varepsilon} \frac{u^2}{(\log u)^{\frac{B}{2}-5A-5-2\varepsilon}}.
\end{align}
\section{Major Arcs} \label{Major Arcs : Sec}
In this section, we estimate the integral
\[\int_{\mathfrak{M}} S(\alpha)^3 e(-\alpha m) \, \dd   \alpha.\]
We write
\begin{align}\label{Major Arc Integral Breaking}
\int_{\mathfrak{M}} S(\alpha)^3 e(-\alpha m) \, \dd   \alpha&=\sum_{q\leqslant P}\sum\limits_{\substack{a=1 \\ (a,q)=1}}^q \int_{\mathfrak{M}(q, a)} S(\alpha)^3 e(-\alpha m) \, \dd   \alpha.
\end{align}
Fix a major arc $\mathfrak{M}(q, a)$ and let $\alpha \in \mathfrak{M}(q, a)$. There exists a reduced fraction $a/q$ with  $1 \leqslant a \leqslant q \leqslant P$ and $(a, q)=1$ such that $|\alpha-a / q| \leqslant 1/Q$. Recall from \eqref{Detecting AP}, we have
\[
S(\alpha) = \frac{1}{q_0} \sum_{r \in \mathcal{R}_0}  \sum_{\ell=0}^{q_0-1} e \left (\frac{- \ell r}{q_0} \right ) \sum_{\substack{u \leqslant k \leqslant 2 u}} \Lambda(k) e \bigg (k \left(\alpha +\frac{\ell}{q_0} \right)\bigg).
\]
Focusing on the innermost sum, we write
\begin{align}\label{Inner Sum Definition}
\tilde{S}(\alpha_{\ell}) = \sum_{\substack{u \leqslant k \leqslant 2 u}} \Lambda(k) e (k \alpha_{\ell}),
\end{align}
where \[
\alpha_{\ell} = \alpha+\frac{\ell}{q_0}, \quad \frac{a}{q}+\frac{\ell}{q_0} = \frac{a_{\ell}}{q_{\ell}} \textrm{ with } (a_{\ell}, q_{\ell})=1, \quad \textrm{and } \bigg \lvert \alpha_{\ell} - \frac{a_{\ell}}{q_{\ell}} \bigg \rvert \leqslant \frac{1}{Q}. 
\]
Let $\beta = \alpha-a/q$. Then we have
\begin{align}\label{Inner Sum Step 1}
\tilde{S}(\alpha_{\ell}) & =\sum_{\substack{u \leqslant k \leqslant 2 u \\
(k, q_{\ell})=1}} \Lambda(k) e\left(\frac{k a_{\ell}}{q_{\ell}}\right) e(k \beta)+\sum_{\substack{u \leqslant k \leqslant 2u \\
(k, q_{\ell})>1}} \Lambda(k) e\left(\frac{k a_{\ell}}{q_{\ell}}\right) e(k \beta).
\end{align}
We detect the exponential phase in \eqref{Inner Sum Step 1} using the following identity
$$
\frac{1}{\varphi(q_{\ell})} \sum_{\chi \bmod q_{\ell}} \tau(\overline{\chi}) \chi(h)= \begin{cases}e(h / q_{\ell}) & \text { if }(h, q_{\ell})=1, \\ 0 & \text { if }(h, q_{\ell})>1 ,\end{cases}
$$
where $\tau(\chi)$ is the Gauss sum. This yields
\begin{align}\label{Detecting Exponential Phase}
\tilde{S}(\alpha_{\ell}) & =\frac{1}{\varphi(q_{\ell})} \sum_{\chi \bmod q_{\ell}} \tau(\overline{\chi}) \chi(a_{\ell}) \sum_{u \leqslant k \leqslant 2 u} \Lambda(k) \chi(k) e(k \beta)+O\bigg(\sum_{p \mid q_{\ell}} \sum_{r \leqslant \log _p (2u)} \Lambda\left(p^r\right)\bigg).
\end{align}
Note that $\sum_{p \mid q_{\ell}} 1=\omega(q_{\ell}) \ll \log q_{\ell} \leqslant \log P \leqslant \log u$. Therefore the error term in \eqref{Detecting Exponential Phase} is $\ll (\log u)^2$. We focus on the first term on the right hand side of \eqref{Detecting Exponential Phase}. Regarding the inner sum over $k$, if we let
\[\psi(t,\chi) = \sum_{k \leqslant t}\chi(k)\Lambda(k),
\]
then by partial summation, we have
\begin{align}\label{Inner Sum Step 2}
\sum_{u \leqslant k \leqslant 2 u} \Lambda(k) \chi(k) e(k \beta) = e(2u \beta) \psi(2u, \chi)-e(u \beta) \psi(u, \chi)-2 \pi i \beta \int_u^{2u} e(t \beta) \psi(t, \chi) \, \dd t.
\end{align}
Applying the Siegel--Walfisz theorem (see \cite[Ch. 22]{D2000}), when $\chi \neq \chi_0$, it follows that
\begin{align}\label{Non-Principal Characters}
\sum_{u \leqslant k \leqslant 2 u} \Lambda(k) \chi(k) e(k \beta) \ll (1+2\lvert \beta \rvert u) u \exp(-c \sqrt{\log u}),
\end{align}
for some absolute constant $c>0$. When $\chi = \chi_0$, let $\psi\left(t, \chi_0\right)=[t]+R(t)$ and
\begin{align}\label{T Beta Definition}
T(\beta)=\sum_{u \leqslant k \leqslant 2u} e(k \beta).
\end{align}
By partial summation,
\begin{align}\label{T Beta Estimate}
T(\beta)=[2u]e(2u \beta)-[u]e(u\beta)-2 \pi i \beta \int_u^{2u} e(t \beta)[t] \, \dd   t .
\end{align}
Then we have
\begin{align}\label{Principal Character}
\sum_{u \leqslant k \leqslant 2u}& \chi_0(k) \Lambda(k) e(k \beta) \notag \\
& =T(\beta)+\left(e(2u \beta) \psi(2u, \chi_0)-e(u \beta) \psi(u, \chi_0)-2 \pi i \beta \int_u^{2u} e(t \beta) \psi(t, \chi_0) \, \dd   t - T(\beta)\right) \notag  \\
&=T(\beta)+e(2u \beta) R(2u)-e(u \beta) R(u)-2 \pi i \beta \int_u^{2u} e(t \beta) R(t) \, \dd   t \notag  \\
& =T(\beta)+O\bigg((1+|\beta| u) u \exp (-c \sqrt{\log u})\bigg).
\end{align}
We have $\tau\left(\chi_0\right)=\mu(q_{\ell})$ and  $|\tau(\chi)| \leqslant q_{\ell}^{1/2}$. Substituting \eqref{Non-Principal Characters} and \eqref{Principal Character} into \eqref{Detecting Exponential Phase}, we obtain
\begin{align*}
\tilde{S}(\alpha_{\ell})=\frac{\mu(q_{\ell})}{\varphi(q_{\ell})} T(\beta)+O\left(q_{\ell}^{1/2}(1+|\beta| u) u \exp (-c \sqrt{\log u})\right) .
\end{align*}
But $q_{\ell} \leqslant qq_0$ and $|\beta| \leqslant 1 / Q$ for $\alpha \in \mathfrak{M}(q, a)$, so we arrive at 
\begin{align}\label{Inner Sum Final Estimate}
\tilde{S}(\alpha_{\ell})=\frac{\mu(q_{\ell})}{\varphi(q_{\ell})} T(\beta)+O\left(u \exp \left(-c_1 \sqrt{\log u}\right)\right),
\end{align}
for some absolute constant $c_1>0$. Taking cubes of both sides of \eqref{Detecting AP}, we write
\begin{align}\label{Taking Cubes}
S(\alpha)^3 &= \frac{1}{q_0^3} \underset{\substack{r_1,r_2,r_3 \in \mathcal{R}_0}}{\sum \sum \sum} \underset{\substack{0\leqslant \ell_1,\ell_2,\ell_3 \leqslant q_0-1}}{\sum \sum \sum}  e \left (\frac{-( \ell_1 r_1+\ell_2r_2+\ell_3r_3)}{q_0} \right ) \tilde{S}(\alpha_{\ell_1}) \tilde{S}(\alpha_{\ell_2}) \tilde{S}(\alpha_{\ell_3}).
\end{align}
From \eqref{Inner Sum Final Estimate}, it follows that
\begin{align}\label{Product of Inner Sum}
\tilde{S}(\alpha_{\ell_1}) \tilde{S}(\alpha_{\ell_2}) \tilde{S}(\alpha_{\ell_3}) = \frac{\mu(q_{\ell_1}) \mu(q_{\ell_2}) \mu(q_{\ell_3})}{\varphi(q_{\ell_1}) \varphi(q_{\ell_2}) \varphi(q_{\ell_3})} T(\beta)^3+O\left(u^3 \exp \left(-c_1 \sqrt{\log u}\right)\right).
\end{align}
Therefore the contribution to the integral in \eqref{Major Arc Integral Breaking} from $\mathfrak{M}(q,a)$ is
\begin{align*}
\frac{1}{q_0^3} \sum_{r_i} \sum_{\ell_i} & e \left (\frac{-( \ell_1 r_1+\ell_2r_2+\ell_3r_3)}{q_0} \right ) \frac{\mu(q_{\ell_1}) \mu(q_{\ell_2}) \mu(q_{\ell_3})}{\varphi(q_{\ell_1}) \varphi(q_{\ell_2}) \varphi(q_{\ell_3})} e\left( -\frac{am}{q}\right) \int_{-1/Q}^{1/Q} T(\beta)^3e(-m\beta) \, \dd  \beta \\
&+O\left(u^2 \exp \left(-c_2 \sqrt{\log u}\right)\right),
\end{align*}
for some absolute constant $c_1>0$, where we have suppressed some notations in the summation for brevity. Summing over all the major arcs $\mathfrak{M}(q,a)$, we see that
\begin{align}
\int_{\mathfrak{M}}S(\alpha)^3 e(-m \alpha) \, d   \alpha \notag  &= \frac{1}{q_0^3} \sum_{r_i} \sum_{\ell_i}e \left (\frac{-( \ell_1 r_1+\ell_2r_2+\ell_3r_3)}{q_0} \right ) \\
&\quad \times \sum_{q \leqslant P} \sum_{\substack{1\leqslant a \leqslant q \\ (a,q)=1 }} \frac{\mu(q_{\ell_1}) \mu(q_{\ell_2}) \mu(q_{\ell_3})}{\varphi(q_{\ell_1}) \varphi(q_{\ell_2}) \varphi(q_{\ell_3})} e\left( -\frac{am}{q}\right)  \int_{-1 / Q}^{1 / Q} T(\beta)^3 e(-m \beta)\, \dd   \beta \notag \\
&\quad+O\left(u^2 \exp \left(-c_3 \sqrt{\log u}\right)\right). \label{Major Arc Contributions}
\end{align}
We now focus on the first term on the right hand side of \eqref{Major Arc Contributions}. We begin by evaluating the inside integral over $\beta$. Noting that the sum $T(\beta)$ is a geometric series, we obtain 
$$
T(\beta) = \frac{e(([2u]+1) \beta)-e(([u]+1) \beta)}{e(\beta)-1} \ll \min \left(u,\|\beta\|^{-1}\right) .
$$
Hence it follows that
$$
\int_{1 / Q}^{1-1 / Q}|T(\beta)|^3 \, \dd   \beta=O\bigg(\int_{1 / Q}^{1 / 2} \beta^{-3} \, \dd   \beta\bigg)=O\left(Q^2\right)=O\left(u^2(\log u)^{-2 B}\right).
$$
This implies that
\begin{align}\label{Extending the Range of Integral}
\int_{-1 / Q}^{1 / Q} T(\beta)^3 e(-m \beta) \, \dd   \beta=\int_0^1 T(\beta)^3 e(-m \beta) \, \dd   \beta+O\left(u^2(\log u)^{-2 B}\right) .
\end{align}
The integral on the right hand side of \eqref{Extending the Range of Integral} equals
\begin{align}\label{Removing Primes}
\int_0^1\bigg(\sum_{u \leqslant k_1, k_2, k_3 \leqslant 2u} e\left(\left(k_1+k_2+k_3\right) \beta\right)\bigg) e(-m \beta) \, \dd   \beta=\sum_{\substack{u \leqslant k_1, k_2, k_3 \leqslant 2u \\ k_1+k_2+k_3=m}} 1,
\end{align}
that is, the number of ways of writing $m$ in the form $m=k_1+k_2+k_3$ with positive integers $k_1, k_2, k_3 \in [u,2u]$. Hence the left hand side of \eqref{Removing Primes} is equal to $\mathcal{H}(m,u)+O(u)$, where
\begin{align}\label{H(m,u) definition}
\mathcal{H}(m,u) = u^2-\frac{1}{2}\bigg( (m-4u)^2+(5u-m)^2 \bigg).
\end{align}
Substituting this in \eqref{Extending the Range of Integral}, we obtain
\begin{align}\label{T_beta Estimate}
\int_{-1 / Q}^{1 / Q} T(\beta)^3 e(-m \beta) d \beta=\mathcal{H}(m,u)+O\left(u^2(\log u)^{-2 B}\right) .
\end{align}
Substituting \eqref{T_beta Estimate} in \eqref{Major Arc Contributions}, we arrive at
\begin{align}\label{Major Arcs Error Term Step 0}
\int_{\mathfrak{M}}S(\alpha)^3 e(-m \alpha) \, \dd   \alpha & = \frac{\mathcal{H}(m,u)}{q_0^3} \sum_{r_i} \sum_{\ell_i} e \left (\frac{-( \ell_1 r_1+\ell_2r_2+\ell_3r_3)}{q_0} \right )\notag \\
&\quad \times \sum_{q \leqslant P} \sum_{\substack{1\leqslant a \leqslant q \\ (a,q)=1 }} \frac{\mu(q_{\ell_1}) \mu(q_{\ell_2}) \mu(q_{\ell_3})}{\varphi(q_{\ell_1}) \varphi(q_{\ell_2}) \varphi(q_{\ell_3})}  e\left( -\frac{am}{q}\right) \notag \\
& \quad+O\bigg(\frac{u^2}{(\log u)^{9A+2B}} \sum_{r_i} \sum_{\ell_i} \sum_{q \leqslant P} \sum_{\substack{1\leqslant a \leqslant q \\ (a,q)=1 }} \frac{1}{\varphi(q_{\ell_1}) \varphi(q_{\ell_2}) \varphi(q_{\ell_3})}\bigg).
\end{align}
To bound the error terms, we study the sum over the $\ell_i$'s. Fix $q$ and assume $q_{\ell_i} \geqslant 2$ for $i=1,2,3$. Using the estimate that $\varphi(q_{\ell_i}) \gg q_{\ell_i} / \log q_{\ell_i}$ for all $q_{\ell_i} \geqslant 2$, we see that
\begin{align}\label{Bounding Error Terms I}
\sum_{q \leqslant P} \sum_{\substack{1\leqslant a \leqslant q \\ (a,q)=1 }} \underset{\substack{0\leqslant \ell_1,\ell_2,\ell_3 \leqslant q_0-1}}{\sum \sum \sum} \frac{1}{\varphi(q_{\ell_1}) \varphi(q_{\ell_2}) \varphi(q_{\ell_3})} \ll_{\varepsilon} (\log \log u)^{3+\varepsilon} \sum_{q \leqslant P} \sum_{\substack{1\leqslant a \leqslant q \\ (a,q)=1 }} \bigg ( \sum_{\ell =0}^{q_0-1} \frac{1}{q_{\ell}}\bigg )^3.
\end{align}
Let $d = (q,q_0)$ and write $q=d\tilde{q}, q_0 = d \tilde{q_0}$ with $(\tilde{q}, \tilde{q_0})=1$. Then
\begin{align}
\sum_{\ell =0}^{q_0-1} \frac{1}{q_{\ell}} &= \frac{1}{qq_0}\sum_{\ell =0}^{q_0-1} (aq_0+\ell q, qq_0) \notag \\
&= \frac{(q,q_0)}{qq_0}\sum_{\ell =0}^{q_0-1} (a\tilde{q_0}+\ell \tilde{q}, d\tilde{q}\tilde{q_0}) \leqslant \frac{(q,q_0)}{qq_0}\sum_{\ell =0}^{q_0-1} d \tilde{q_0} \leqslant \frac{q_0^2}{q}. \label{GCD Sum}
\end{align}
Substituting \eqref{GCD Sum} in \eqref{Bounding Error Terms I}, we obtain
\begin{align}\label{Bounding Error Terms II}
\sum_{q \leqslant P} \sum_{\substack{1\leqslant a \leqslant q \\ (a,q)=1 }} \underset{\substack{0\leqslant \ell_1,\ell_2,\ell_3 \leqslant q_0-1}}{\sum \sum \sum} \frac{1}{\varphi(q_{\ell_1}) \varphi(q_{\ell_2}) \varphi(q_{\ell_3})} &\ll_{\varepsilon} (\log u)^{18A} (\log \log u)^{3+\varepsilon} \sum_{q \leqslant P} \frac{1}{q^2} \notag \\
&\ll_{\varepsilon} (\log u)^{18A}(\log \log u)^{3+\varepsilon},
\end{align}
where the sum is over those tuples $(\ell_1,\ell_2,\ell_3)$ for which $q_{\ell_i} \geqslant 2$ for $i=1,2,3$. Suppose without loss of generality, $q_{\ell_1}=1$. This implies $q$ must divide $q_0$. Therefore the number of possible choices for $q$ is $\ll_{\varepsilon} (\log u)^{\varepsilon}$ for any $\varepsilon>0$. Hence combining this argument and from \eqref{Bounding Error Terms II}, we see that the error term in \eqref{Major Arcs Error Term Step 0} is 
\[\ll_{\varepsilon} u^2(\log u)^{12A-B+\varepsilon}.\]
Hence it remains to study the expression $\mathcal{H}(m,u)\mathfrak{S}(m,q_0, \mathcal{R}_0)$, where $\mathfrak{S}(m,q_0, \mathcal{R}_0)$ is the arithmetic factor given by
\begin{align}
\mathfrak{S}(m,q_0, \mathcal{R}_0) = \frac{1}{q_0^3} \sum_{q \leqslant P} \sum_{\substack{1\leqslant a \leqslant q \\ (a,q)=1 }} e\left( -\frac{am}{q}\right) \sum_{r_i} \sum_{\ell_i} e \left (\frac{-( \ell_1 r_1+\ell_2r_2+\ell_3r_3)}{q_0} \right )  \frac{\mu(q_{\ell_1}) \mu(q_{\ell_2}) \mu(q_{\ell_3})}{\varphi(q_{\ell_1}) \varphi(q_{\ell_2}) \varphi(q_{\ell_3})}.\label{arithmetic factor}
\end{align}
Our goal across the next two sections is to show that $\mathfrak{S}(m,q_0, \mathcal{R}_0) \neq 0$ when $m$ is odd.
\section{Arithmetic Factor: Initial Decomposition}\label{sec: The arithmetic factor}
\subsection{Initial Steps} Let us define
\begin{align}\label{Defining f}
f(q) := \frac{1}{q_0^3}\sum_{\substack{1\leqslant a \leqslant q \\ (a,q)=1 }} e\left( -\frac{am}{q}\right) \sum_{r_i} \sum_{\ell_i} e \left (\frac{-( \ell_1 r_1+\ell_2r_2+\ell_3r_3)}{q_0} \right )  \frac{\mu(q_{\ell_1}) \mu(q_{\ell_2}) \mu(q_{\ell_3})}{\varphi(q_{\ell_1}) \varphi(q_{\ell_2}) \varphi(q_{\ell_3})}.
\end{align}
Note that if $q>P$ with $B$ chosen sufficiently large in terms of $A$, then $q$ doesn't divide $q_0$. Hence $q_{\ell}$ is never equal to 1. Therefore again using \eqref{GCD Sum}, it follows that
\begin{align}
\sum_{q >P} \lvert f(q) \rvert \ll \frac{1}{(\log u)^{B-27A-1}}.\label{convergence of arithmetic factor}
\end{align}
Thus we deduce that 
\[
\mathfrak{S}(m,q_0, \mathcal{R}_0) = \sum_{q=1}^{\infty}f(q)+O((\log u)^{12A-B+\varepsilon}).
\] 
\subsection{Arithmetic Arguments involving GCDs} For any $q \in \mathbb{N}$, we can decompose $q$ uniquely as
\begin{align}\label{q Decomposition 1}
q = \tilde{q}\hat{q},
\end{align}
where 
\begin{align}\label{q Decomposition 2}
\tilde{q} = \prod_{\substack{p|q_0\\p^{v}||q}} p^{v} \text{ and  } \hat{q} = \frac{q}{\tilde{q}}.
\end{align}
Note that $(\hat{q},q_0)=1$ and $(\hat{q},\tilde{q})=1$. Let $d=(\tilde{q},q_0)$. Recall that $q_0=p_1p_2\cdots p_k$, where $p_i$ are distinct primes $<z$. Therefore it follows that
\begin{align}\label{GCD Formula}
d = \prod_{\substack{p|q_0, p|\tilde{q}}} p.
\end{align}
Let us write $\tilde{q} = \prod_{i \in \mathcal{I}} p_i^{v_i}$ where $\mathcal{I} \subseteq \{1,2,\dots,k\}$ and $v_i\geqslant 1$. Then $d = \prod_{i \in \mathcal{I}} p_i$ and 
\begin{align}
q_\ell &= \frac{qq_0}{(qq_0,aq_0+\ell q)} = \frac{\tilde{q}\hat{q}q_0}{(\tilde{q}\hat{q}q_0,aq_0+\ell\tilde{q}\hat{q})}=\frac{\tilde{q}\hat{q}q_0}{(\tilde{q}q_0,aq_0+\ell\tilde{q}\hat{q})}=\hat{q}\frac{\tilde{q}q_0}{(\tilde{q}q_0,aq_0+\ell\hat{q}\tilde{q})}.\label{ql breakdown}
\end{align}
Moreover, we have 
\[
\bigg(\hat{q},\frac{\tilde{q}q_0}{(\tilde{q}q_0,aq_0+\ell\hat{q}\tilde{q})}\bigg)=1.
\]
Using \eqref{ql breakdown}, we may rewrite $f(q)$ as
\begin{align*}
& f(q) = \frac{1}{q_0^3}\sum_{\substack{1\leqslant a \leqslant \tilde{q}\hat{q} \\ (a,\tilde{q}\hat{q})=1 }} e\bigg( -\frac{am}{\tilde{q}\hat{q}}\bigg) \sum_{r_i} \sum_{\ell_i} e \bigg (\frac{-( \ell_1 r_1+\ell_2r_2+\ell_3r_3)}{q_0} \bigg ) \frac{\mu(\hat{q})}{\varphi(\hat{q})^3} \prod_{i=1}^3 \frac{\mu}{\varphi} \bigg(\frac{\tilde{q}q_0}{(\tilde{q}q_0,aq_0+\ell_i\hat{q}\tilde{q})}\bigg).
\end{align*}
Since $\ell_i$ runs through the complete residue class modulo $q_0$, $\ell_i\hat{q}$ also runs through the complete residue class modulo $q_0$. This implies that
\begin{align}
f(q)&=\frac{1}{q_0^3}\frac{\mu(\hat{q})}{\varphi(\hat{q})^{3}}\sum_{\substack{1\leqslant a \leqslant \tilde{q}\hat{q} \notag\\ (a,\tilde{q})=1\\(a,\hat{q})=1 }} e\bigg( -\frac{am}{\tilde{q}\hat{q}}\bigg)
\sum_{r_i} \sum_{\ell_i} e \bigg (\frac{-( \ell_1 r_1+\ell_2r_2+\ell_3r_3)}{\hat{q}q_0} \bigg ) \prod_{i=1}^3 \frac{\mu}{\varphi} \bigg(\frac{\tilde{q}q_0}{(\tilde{q}q_0,aq_0+\ell_i\tilde{q})}\bigg) \notag\\
&=\frac{1}{q_0^3}\frac{\mu(\hat{q})}{\varphi(\hat{q})^{3}}\sum_{\substack{1\leqslant a \leqslant \tilde{q}\hat{q} \\ (a,\tilde{q})=1\\(a,\hat{q})=1 }} e\bigg( -\frac{am}{\tilde{q}\hat{q}}\bigg)
\bigg(\sum_{\ell=0}^{q_0-1} \frac{\mu}{\varphi} \bigg(\frac{\tilde{q}q_0}{(\tilde{q}q_0,aq_0+\ell\tilde{q})}\bigg) \sum_{r\in \mathcal{R}_0} e \left (\frac{- \ell r}{\hat{q}q_0} \right ) \bigg)^3. \label{f(q) arithmetic factor}
\end{align}
Now let us examine the M\"{o}bius factor in the inner sum in \eqref{f(q) arithmetic factor}. We write
\begin{align}
   \mu\bigg(\frac{\tilde{q}q_0}{(\tilde{q}q_0,aq_0+\ell\tilde{q})}\bigg)&=\mu \bigg(\frac{\tilde{q}q_0/d}{(\tilde{q}q_0/d,aq_0/d+\ell \tilde{q}/d}\bigg),\label{fraction}
\end{align}
which is nonzero if and only if 
$\tilde{q}/d$ divides $a$. Since $(a,q)=1$, we have $(a,\tilde{q})=1$. So the only way for this to happen is when $d = \tilde{q}$, which means that $\tilde{q}$ divides $q_0$. We now evaluate $f(q)$ in two cases. \\

\noindent \textbf{Case 1: $\tilde{q}$ does not divides $q_0$.} By the discussion above, this case will force the the M\"{o}bius factor in the inner sum in \eqref{f(q) arithmetic factor} to be zero for all $a$. Then $f(q)=0$.\\

\noindent \textbf{Case 2: $\tilde{q}$ divides $q_0$.} In this case $d = \tilde{q}$. Then we have 
\begin{align}
     \frac{\tilde{q}q_0}{(\tilde{q}q_0,aq_0+\ell\tilde{q})}=\frac{q_0}{(q_0,aq_0/\tilde{q}+\ell)}.\label{case 2 breakdown of factors}
\end{align}
Substituting \eqref{case 2 breakdown of factors} into the representation of $f(q)$ in \eqref{f(q) arithmetic factor}, we get
\begin{align}
f(q)& = \frac{1}{q_0^3}\frac{\mu(\hat{q})}{\varphi(\hat{q})^{3}}\sum_{\substack{1\leqslant a \leqslant \tilde{q}\hat{q} \\ (a,\tilde{q})=1\\(a,\hat{q})=1 }} e\bigg( -\frac{am}{\tilde{q}\hat{q}}\bigg)
\bigg(\sum_{\ell=0}^{q_0-1} \frac{\mu}{\varphi} \bigg(\frac{q_0}{(q_0,aq_0/\tilde{q}+\ell)}\bigg)\sum_{r\in \mathcal{R}_0} e \bigg (\frac{- \ell r}{\hat{q}q_0} \bigg ) \bigg)^3. \label{f(q) case 2}
\end{align}
Using \eqref{f(q) case 2}, we can evaluate $f$ at $\tilde{q}$. We see that if $\tilde{q}$ divides $q_0$ then
\begin{align*}
f(\tilde{q}) &=\frac{1}{q_0^3}\sum_{\substack{1\leqslant a \leqslant \tilde{q} \\ (a,\tilde{q})=1 }} e\bigg( -\frac{am}{\tilde{q}}\bigg) \bigg( \sum_{\ell=0}^{q_0-1} \frac{\mu}{\varphi} \bigg(\frac{q_0}{(q_0,aq_0/\tilde{q}+\ell)}\bigg)\sum_{r\in \mathcal{R}_0}e \bigg (\frac{-\ell r}{q_0} \bigg ) \bigg)^3.
\end{align*}
Similarly, we get
\begin{align}
f(\hat{q}) =\frac{1}{q_0^3}\frac{\mu(\hat{q})}{\varphi(\hat{q})^{3}}\sum_{\substack{1\leqslant a \leqslant \hat{q} \\ (a,\hat{q})=1 }} e\bigg( -\frac{am}{\hat{q}}\bigg)
\bigg(\sum_{r\in \mathcal{R}_0} \sum_{\ell=0}^{q_0-1} \frac{\mu}{\varphi} \bigg(\frac{q_0}{(q_0,\ell)}\bigg) e \bigg (\frac{- \ell r}{q_0} \bigg ) \bigg)^3.\label{f(q hat)}
\end{align}
Moreover, we also have
\begin{align}
    f(1) &=\frac{1}{q_0^3}\bigg(\sum_{\ell=0}^{q_0-1} \frac{\mu}{\varphi} \bigg(\frac{q_0}{(q_0,\ell)}\bigg)\sum_{r\in \mathcal{R}_0}e\left(\frac{-\ell r}{q_0}\right)\bigg)^3.\label{f(1) equation}
\end{align}
\subsection{Decomposition of $f$ using Dirichlet Convolution} Our objective here is to decompose $f$ as 
\[
f(q) = (g*h)(q).
\]
where $g$ and $h$ are arithmetic functions. Furthermore if $\sum_{v=1}^\infty h(v)<\infty$, then by \eqref{convergence of arithmetic factor}, it follows that 
\begin{align}\label{Convergence of Dirichlet Factors}
\sum_{q=1}^\infty f(q) = \bigg(\sum_{w=1}^\infty  g(w)\bigg)\bigg(\sum_{v=1}^\infty h(v)\bigg). 
\end{align}
Let us construct our functions $g$ and $h$. We define $g$ as follows:
\begin{align}
g(w) = 
    \begin{cases}
        0\quad & \text{if }w\nmid q_0,\\
        \frac{1}{q_0^3}\sum_{\substack{1\leqslant a \leqslant w \\ (a,w)=1 }} e\left( -\frac{am}{w}\right) \bigg( \sum_{\ell=0}^{q_0-1} \frac{\mu}{\varphi} (\frac{q_0}{(q_0,aq_0/w+\ell)})\sum_{r\in \mathcal{R}_0}e (\frac{-\ell r}{q_0}) \bigg)^3 \quad &\text{else}.
    \end{cases}\label{g(w)}
\end{align}
Let $h$ be defined as follows:
\begin{align}
h(v) = \begin{cases}
        0\quad & \text{if }(v,q_0)>1,\\
        \frac{\mu(v)}{\varphi(v)^3}\sum_{\substack{1\leqslant a \leqslant v \\ (a,v)=1 }} e\left( -\frac{am}{v}\right) \quad &\text{else}.
    \end{cases} \label{h(v)}
\end{align}
By the discussion in \cite[Ch. 25]{D2000}, $h$ is multiplicative and $\sum_{v=1}^\infty h(v)$ converges absolutely. Thus we get
\begin{align}
\sum_{v=1}^\infty h(v) &= \prod_{p\text{ prime}}\sum_{j=0}^\infty h(p^j) \notag\\
    &=\prod_{\substack{p\text{ prime}\\p|q_0}}(\sum_{j=0}^\infty h(p^j))\prod_{\substack{p\text{ prime}\\p\nmid q_0}}(\sum_{j=0}^\infty h(p^j)) \notag \\
    &=\prod_{\substack{p\nmid q_0\\p|m}}\left(1-\frac{1}{(p-1)^2}\right)\prod_{\substack{p\nmid q_0\\p\nmid m}}\left(1+\frac{1}{(p-1)^3}\right). \label{Euler Product for h}
\end{align}
We write
\begin{align}
(g*h)(q)= \sum_{c|q} g(c)h\left(\frac{q}{c}\right ) =\sum_{\substack{c|q, c|q_0\\(q/c,q_0)=1}}g(c)h\left(\frac{q}{c}\right )=g(d)h(q/d),\label{f(q)=g(c)h(q/c)}
\end{align}
where $d=(\tilde{q},q_0)$. If $\tilde{q}$ doesn't divide $q_0$, then $d \neq \tilde{q}$ which implies that $h(q/d)=0$. Therefore we have $f(q)=(g*h)(q)$ in this case. If $\tilde{q}$ divides $q_0$, then $d= \tilde{q}$ which implies that
\begin{align}\label{f as a Convolution}
(g*h)(q) &= g(\tilde{q})h(\hat{q})\notag \\
& = \frac{1}{q_0^3}\frac{\mu(\hat{q})}{\varphi(\hat{q})^{3}}\sum_{\substack{1\leqslant a_1 \leqslant \tilde{q} \\ 1\leqslant a_2 \leqslant \hat{q} \\ (a_1,\tilde{q})=1  \\ (a_2,\hat{q})=1 }} e\bigg( -\frac{a_1m}{\tilde{q}}-\frac{a_2m}{\hat{q}}\bigg)
\bigg(\sum_{\ell=0}^{q_0-1} \frac{\mu}{\varphi} \bigg(\frac{q_0}{(q_0,a_1q_0/\tilde{q}+\ell)}\bigg)\sum_{r\in \mathcal{R}_0} e \bigg (\frac{- \ell r}{\hat{q}q_0} \bigg ) \bigg)^3 \notag \\
&= \frac{1}{q_0^3}\frac{\mu(\hat{q})}{\varphi(\hat{q})^{3}}\sum_{\substack{1\leqslant a \leqslant \tilde{q}\hat{q} \\ (a,\tilde{q})=1\\(a,\hat{q})=1 }} e\bigg( -\frac{am}{\tilde{q}\hat{q}}\bigg)
\bigg(\sum_{\ell=0}^{q_0-1} \frac{\mu}{\varphi} \bigg(\frac{q_0}{(q_0,aq_0/\tilde{q}+\ell)}\bigg)\sum_{r\in \mathcal{R}_0} e \bigg (\frac{- \ell r}{\hat{q}q_0} \bigg ) \bigg)^3 =f(q).
\end{align}
Therefore it follows that $f=g*h$. By \eqref{Euler Product for h}, we already have 
\[
 \sum_{v=1}^\infty h(v) > 0
\]
for any odd integer $m$. Hence it suffices to study the function $g$ which we do in the following section.
\section{Arithmetic Factor: Further Decomposition}\label{sec: Arithmetic factor II}
\subsection{Initial Steps} Recall that in \eqref{g(w)}, we defined $g(w)=0$ if $w\nmid q_0$. We begin by writing
\begin{align}\label{Define G}
G(q_0) := \sum_{w=1}^\infty g(w) = \frac{1}{q_0^3}\sum_{w|q_0} \sum_{\substack{1\leqslant a \leqslant w \\ (a,w)=1 }} e\bigg(-\frac{am}{w}\bigg) \bigg( \sum_{\ell=0}^{q_0-1} \frac{\mu}{\varphi} \bigg(\frac{q_0}{(q_0,aq_0/w+\ell)}\bigg)\sum_{r\in \mathcal{R}_0}e \bigg (\frac{-\ell r}{q_0} \bigg ) \bigg)^3.
\end{align}
Our objective is to show that $G(q_0)>0$. We denote $g(w)$ by $g(w,q_0)$ to show the dependence on $q_0$. Expanding the expression for $g(w,q_0)$, we obtain
\begin{align}
g(w,q_0) &= \frac{1}{q_0^3} \sum_{\substack{1\leqslant a \leqslant w \\ (a,w)=1 }} e\left( -\frac{am}{w}\right) \underset{r_1, r_2,r_3\in \mathcal{R}_0}{\sum \sum \sum} \underset{0 \leqslant \ell_1, \ell_2, \ell_3 \leqslant q_0-1}{\sum \sum \sum} e \left (\frac{-\ell_1 r_1}{q_0} \right )e \left (\frac{-\ell_2 r_2}{q_0} \right )e \left (\frac{-\ell_3 r_3}{q_0} \right )\notag\\
&\quad \times \frac{\mu}{\varphi} \bigg(\frac{q_0}{(q_0,aq_0/w+\ell_1)}\bigg) \frac{\mu}{\varphi} \bigg(\frac{q_0}{(q_0,aq_0/w+\ell_2)}\bigg) \frac{\mu}{\varphi} \bigg(\frac{q_0}{(q_0,aq_0/w+\ell_3)}\bigg).\label{g(w,q_0)}
\end{align}
Recall that $q_0 = p_1p_2\cdots p_k$. For $1 \leqslant i \leqslant k$, define
\begin{align*}
u_i\equiv
    \begin{cases}
        1 \mod p_i\\
        0 \mod p_j \quad\text{ for $j\neq i$}. 
\end{cases}
\end{align*}
Then by the Chinese remainder theorem, for $i=1,2,3$, and $j=1,2,\dots, k$, there exist unique $a_j,$ $\ell_{ij}\in \Z,$ and $r_{ij}\in \mathcal{R}_j$ such that
\begin{align}
\ell_{i} &= \ell_{i1} u_1+\ell_{i2} u_2+\cdots +\ell_{ik} u_k,\notag \\
a &= a_1u_1+a_2u_2+\cdots+a_ku_k, \notag\\
r_i &= r_{i1}u_1+r_{i2}u_2+\cdots+r_{ik}u_k.\label{base decomposition}
\end{align}
 Note that 
 \begin{align}
     \ell_i&\equiv \ell_{i1}\bmod p_1,\notag\\
     \ell_i&\equiv \ell_{i2}\bmod p_2,\notag\\
     &\vdots\notag\\
     \ell_i&\equiv \ell_{ik}\bmod p_k,\label{base decomposition modulos}
 \end{align}
for $i=1,2,3$. An analogous set of equations hold for $r_i$ as well. For $b\in \mathbb{N}$, define 
\begin{align}\label{Combinatorial Function}
\mathcal{C}(b) := \sum_{t_1=0}^3\cdots\sum_{t_b=0}^3\frac{1}{\varphi(p_1)^{t_1}\cdots \varphi(p_b)^{t_b}} \dbinom{3}{t_1}\dbinom{3}{t_2} \cdots \dbinom{3}{t_b}.
\end{align}
We prove the following auxiliary lemma. 
\begin{lem}\label{lemma for c(k)}
For any $b\in \mathbb{N}$, we have
\begin{align}\label{Formula for c(k)}
\mathcal{C}(b) = \frac{(p_1p_2\cdots p_b)^3}{\varphi(p_1p_2\cdots p_b)^3}.
\end{align}
\end{lem}
\begin{proof}
We will prove it by induction. For the base case, when $b=1$, we have
\begin{align*}
\mathcal{C}(1)&=\sum_{t_1=0}^3\frac{1}{(p_1-1)^{t_1}}\dbinom{3}{t_1}=1+\frac{3}{p_1-1}+\frac{3}{(p_1-1)^2}+\frac{1}{(p_1-1)^3}=\frac{p_1^3}{\varphi(p_1)^3},
\end{align*}
Now assume \eqref{Formula for c(k)} holds for $\mathcal{C}(j)$ for some $j\in \mathbb{N}$. Then we obtain 
\begin{align*}
\mathcal{C}(j+1)&=\sum_{t_1=0}^3\cdots\sum_{t_j=0}^3\sum_{t_{j+1}=0}^3\frac{1}{(p_1-1)^{t_1}\cdots (p_k-1)^{t_j}(p_{j+1}-1)^{t_{j+1}}} \dbinom{3}{t_1}\dbinom{3}{t_2} \cdots \dbinom{3}{t_{j}}\dbinom{3}{t_{j+1}}\\
&=\mathcal{C}(j) \sum_{t_{j+1}=0}^3\frac{1}{(p_{j+1}-1)^{t_{j+1}}}\dbinom{3}{t_{j+1}} =\frac{p_{j+1}^3}{(p_{j+1}-1)^3}\mathcal{C}(j)=\frac{(p_1p_2\cdots p_{j+1})^3}{\varphi(p_1p_2\cdots p_{j+1})^3}.
    \end{align*}
Hence by induction, the proof follows.
\end{proof}
\subsection{Decomposition of $G$} We are now ready to prove the following key lemma which provides the decomposition of $G(q_0)$.
\begin{lem}\label{breakdown og G(q0)}
Let $k \in \mathbb{N}$ and suppose $q_0=p_1p_2\cdots p_k$ where $p_1,p_2,\dots, p_k$ denote the first $k$ primes. Then we have
\[G(q_0)  = \prod_{j=1}^k G(p_j).\]
\end{lem}
\begin{proof}
We will proceed by induction. When $k=1$, the claim is true by definition. For clarity, we show the subtle arguments involving the case $k=2$. By \eqref{Define G}, we write
\begin{align}\label{Breaking G}
G(p_1p_2)= g(1,p_1p_2)+g(p_1,p_1p_2)+g(p_2,p_1p_2)+g(p_1p_2,p_1p_2).
\end{align}
By the definition of $g$ in \eqref{g(w,q_0)}, we get
\begin{align*}
    g(1,p_1p_2) &= \frac{1}{(p_1p_2)^3}  \sum_{r_1\in \mathcal{R}_0}\sum_{r_2\in \mathcal{R}_0}\sum_{r_3\in \mathcal{R}_0}\sum_{\ell_1=0}^{p_1p_2-1}\sum_{\ell_2=0}^{p_1p_2-1}\sum_{\ell_3=0}^{p_1p_2-1}e \left (\frac{-\ell_1 r_1}{p_1p_2} \right )e \left (\frac{-\ell_2 r_2}{p_1p_2} \right )e \left (\frac{-\ell_3 r_3}{p_1p_2} \right )\\
    &\quad \times\frac{\mu}{\varphi}\bigg(\frac{p_1p_2}{(p_1p_2,p_1p_2+\ell_1)}\bigg)\frac{\mu}{\varphi}\bigg(\frac{p_1p_2}{(p_1p_2,p_1p_2+\ell_2)}\bigg)\frac{\mu}{\varphi}\bigg(\frac{p_1p_2}{(p_1p_2,p_1p_2+\ell_3)}\bigg).
\end{align*}
Using \eqref{base decomposition} and the discussions involving \eqref{base decomposition modulos}, we can write
\begin{align*}
    g(1,p_1p_2)
    &=  \frac{1}{p_1^3p_2^3} \underset{r_{11}, r_{21}, r_{31} \in \mathcal{R}_1}{\sum \sum \sum} \hspace{0.1cm}\underset{r_{12}, r_{22}, r_{32} \in \mathcal{R}_2}{\sum \sum \sum}  \hspace{0.1cm} \underset{0 \leqslant \ell_{11}, \ell_{21}, \ell_{31} \leqslant p_1-1 }{\sum \sum \sum} 
 \hspace{0.1cm} \underset{0 \leqslant \ell_{12}, \ell_{22}, \ell_{32} \leqslant p_2-1 }{\sum \sum \sum} e \bigg (\frac{-\ell_{11} r_{11}}{p_1p_2} \bigg )e\left (\frac{-\ell_{12} r_{12}}{p_1p_2} \right )\\
&\times e \left (\frac{-\ell_{21} r_{21}}{p_1p_2} \right ) e \left (\frac{-\ell_{22} r_{22}}{p_1p_2} \right )e \left (\frac{-\ell_{31} r_{31}}{p_1p_2} \right )e \left (\frac{-\ell_{32} r_{32}}{p_1p_2} \right ) \frac{\mu}{\varphi} \bigg(\frac{p_1p_2}{(p_1p_2,p_1p_2+\ell_{11}u_1+\ell_{12}u_2)}\bigg)\\
&\times \frac{\mu}{\varphi} \bigg(\frac{p_1p_2}{(p_1p_2,p_1p_2+\ell_{21}u_1+\ell_{22}u_2)}\bigg)\frac{\mu}{\varphi} \bigg(\frac{p_1p_2}{(p_1p_2,p_1p_2+\ell_{31}u_1+\ell_{32}u_2)}\bigg).\\
\end{align*}
Since $\mu$ and $\varphi$ are multiplicative, we rewrite $g(1,p_1p_2)$ as
\begin{align}
    g(1,p_1p_2)
    &=  \frac{1}{p_1^3p_2^3} \underset{r_{11}, r_{21}, r_{31} \in \mathcal{R}_1}{\sum \sum \sum} \hspace{0.1cm}\underset{r_{12}, r_{22}, r_{32} \in \mathcal{R}_2}{\sum \sum \sum}  \hspace{0.1cm} \underset{0 \leqslant \ell_{11}, \ell_{21}, \ell_{31} \leqslant p_1-1 }{\sum \sum \sum} 
 \hspace{0.1cm} \underset{0 \leqslant \ell_{12}, \ell_{22}, \ell_{32} \leqslant p_2-1 }{\sum \sum \sum} e \left (\frac{-\ell_{11} r_{11}}{p_1p_2} \right )e\left (\frac{-\ell_{12} r_{12}}{p_1p_2} \right )\notag\\
&\times e \left (\frac{-\ell_{21} r_{21}}{p_1p_2} \right ) e \left (\frac{-\ell_{22} r_{22}}{p_1p_2} \right )e \left (\frac{-\ell_{31} r_{31}}{p_1p_2} \right )e \left (\frac{-\ell_{32} r_{32}}{p_1p_2} \right ) \frac{\mu}{\varphi}\bigg(\frac{p_1}{(p_1,\ell_{11}u_1)}\bigg)\frac{\mu}{\varphi}\bigg(\frac{p_2}{(p_2,\ell_{12}u_2)}\bigg) \notag\\
&\times \frac{\mu}{\varphi}\bigg(\frac{p_1}{(p_1,\ell_{21}u_1)}\bigg)\frac{\mu}{\varphi}\bigg(\frac{p_2}{(p_2,\ell_{22}u_2)}\bigg)\frac{\mu}{\varphi}\bigg(\frac{p_1}{(p_1,\ell_{31}u_1)}\bigg)\frac{\mu}{\varphi}\bigg(\frac{p_2}{(p_2,\ell_{32}u_2)}\bigg).\label{g(1,p1p2) further decomposition}
\end{align}
For $j=1,2$ and $i=1,2,3$, if $p_j\mid (p_1p_2+\ell_i)$, then by \eqref{base decomposition modulos}, $p_j\mid \ell_{ij}$. Moreover, since $\ell_{ij}$ runs through the complete residue class of $p_j$, if $p_j\mid \ell_{ij}$, $\ell_{ij} = 0$ is the only solution. Therefore for each $i$ and $j$, we can divide the sum over $\ell_{ij}$ into two cases, $\ell_{ij}=0$ and $\ell_{ij}\neq 0$. If $\ell_{ij}=0$, 
\begin{align}
    e\left (\frac{-\ell_{ij} r_{ij}}{p_1p_2} \right ) = 1\quad\text{and}\quad \frac{\mu}{\varphi}\left(\frac{p_j}{(p_j,\ell_{ij}u_j)}\right)=1.\label{l_ij = 0}
\end{align}
On the other hand, if $\ell_{ij}\neq 0$, then $p_j\nmid\ell_{ij}$, so
\begin{align}
    \frac{\mu}{\varphi}\left(\frac{p_j}{(p_j,\ell_{ij}u_j)}\right)=  -\frac{1}{\varphi(p_j)}.\label{l_ij neq 0 part 1}
\end{align}
Since $r_{ij}$ and $q_0/p_j$ are both coprime to $p_j$, we deduce that
\begin{align}
    \sum_{\ell_{ij}=1}^{p_j-1}e\left (\frac{-\ell_{ij} r_{ij}}{p_1p_2} \right ) =  - e\left (\frac{-\ell_{ij} r_{ij}}{p_1p_2} \right ).\label{l_ij neq 0 part 2}
\end{align}
Observe that when we substitute \eqref{l_ij neq 0 part 1} and \eqref{l_ij neq 0 part 2} back into \eqref{g(1,p1p2) further decomposition}, the sign in front of the exponential sum and $\mu/\varphi$ cancel out. Since each $\ell_{ij}$ can be divided into the above two cases, there are in total $4^3=64$ cases, and so using \eqref{l_ij = 0}-\eqref{l_ij neq 0 part 2}, we get
\begin{align}\label{gp1p2 final form 1}
    g(1,p_1p_2)    
    &=\frac{\mathcal{C}(2) }{p_1^3p_2^3}\underset{r_{11},r_{21},r_{31}\in \mathcal{R}_1}{\sum\sum\sum}\hspace{0.1cm}\underset{r_{12},r_{22},r_{32}\in \mathcal{R}_2}{\sum\sum\sum} 1=\frac{1}{\varphi(p_1p_2)^3}\mathcal{R}_1^3\mathcal{R}_2^3=\frac{1}{\varphi(p_1p_2)^3} \mathcal{R}_0^3,
\end{align}
where the second equality follows from Lemma \ref{lemma for c(k)}. Similarly as above, we have
\begin{align}\label{g(p1p2)}
    g(p_1, p_1p_2) 
    &=\frac{1}{(p_1p_2)^3} \sum_{a=1}^{p_1-1}e\left(\frac{-am}{p_1}\right)\underset{r_{11},r_{21},r_{31}\in \mathcal{R}_1}{\sum\sum\sum}\hspace{0.1cm}\underset{r_{12},r_{22},r_{32}\in \mathcal{R}_2}{\sum\sum\sum} \hspace{0.1cm} \underset{0 \leqslant \ell_{11}, \ell_{21}, \ell_{31} \leqslant p_1-1 }{\sum \sum \sum} 
 \hspace{0.1cm} \underset{0 \leqslant \ell_{12}, \ell_{22}, \ell_{32} \leqslant p_2-1 }{\sum \sum \sum} \notag \\
    &\quad \times e \left (\frac{-\ell_{11} r_{11}}{p_1p_2} \right ) e\left (\frac{-\ell_{12} r_{12}}{p_1p_2} \right )e \left (\frac{-\ell_{21} r_{21}}{p_1p_2} \right ) e \left (\frac{-\ell_{22} r_{22}}{p_1p_2} \right )e \left (\frac{-\ell_{31} r_{31}}{p_1p_2} \right ) e \left (\frac{-\ell_{32} r_{32}}{p_1p_2} \right ) \notag\\
    &\quad \times
\frac{\mu}{\varphi}\bigg(\frac{p_1p_2}{(p_1p_2,ap_2+\ell_{11}u_1+\ell_{12}u_2)}\bigg)\frac{\mu}{\varphi}\bigg(\frac{p_1p_2}{(p_1p_2,ap_2+\ell_{21}u_1+\ell_{22}u_2)}\bigg) \notag\\
&\quad \times\frac{\mu}{\varphi}\bigg(\frac{p_1p_2}{(p_1p_2,ap_2+\ell_{31}u_1+\ell_{32}u_2)}\bigg).
\end{align}
Note that if $p_2\mid (ap_2+\ell_i)$, then $p_2\mid \ell_{i2}$, and if $p_1\mid (ap_2+\ell_i)$, then $\ell_{i1}\equiv -ap_2\bmod p_1$. Since $\ell_{i1}$ runs through the entire residue class modulo $p_1$, only one such $\ell_{i1}$ could satisfy the congruence condition. The same argument forces $\ell_{i2}$ to be 0. When $\ell_{i1}\equiv -ap_2\bmod p_1$, we have
\begin{align}
e\left (\frac{-\ell_{i1} r_{i1}}{p_1p_2} \right ) = e\left (\frac{a r_{i1}}{p_1} \right )\quad\text{and}\quad \frac{\mu}{\varphi}\left(\frac{p_1}{(p_1,\ell_{i1}u_1)}\right)=1.\label{l_ij = -ap2}
\end{align}
On the other hand, when $\ell_{i1}\not\equiv -ap_2\bmod p_1$, then we deduce that
\begin{align}
\sum_{\substack{\ell_{i1}=0\\e_{i1}\not\equiv -ap_2\bmod p_1}}^{p_1-1}e\left (\frac{-\ell_{i1} r_{i1}}{p_1p_2} \right ) = -e\left (\frac{a r_{i1}}{p_1} \right )\quad\text{and}\quad\frac{\mu}{\varphi}\left(\frac{p_1}{(p_1,\ell_{i1}u_1)}\right)=-\frac{1}{\phi(p_1)}.\label{l_ij != -ap2} 
\end{align}
Substituting \eqref{l_ij = -ap2} and \eqref{l_ij != -ap2} in \eqref{g(p1p2)}, we arrive at
\begin{align}\label{gp1p2 final form 2}
g(p_1,p_1p_2)&= \frac{\mathcal{C}(2) }{p_1^3p_2^3}\underset{r_{11},r_{21},r_{31}\in \mathcal{R}_1}{\sum\sum\sum}\hspace{0.1cm}\underset{r_{12},r_{22},r_{32}\in \mathcal{R}_2}{\sum\sum\sum}\sum_{a=1}^{p_1-1}e \left (\frac{a(r_{11}+r_{21}+r_{31}-m)}{p_1} \right) \notag \\
&=\frac{\mathcal{R}_2^3}{\varphi(p_1p_2)^3} \bigg(-\mathcal{R}_1^3+p_1 \sum_{\substack{r_{11},r_{21},r_{31}\\r_{11}+r_{21}+r_{31}\equiv m\bmod p_1}} 1\bigg) \notag \\
&= \frac{1}{\varphi(p_1p_2)^3} \bigg(-\mathcal{R}_0^3+p_1\mathcal{R}_2^3\sum_{\substack{r_{11},r_{21},r_{31}\\r_{11}+r_{21}+r_{31}\equiv m\bmod p_1}} 1\bigg).
\end{align}
By symmetry, we also have
\begin{align}\label{gp1p2 final form 3}
    g(p_2, p_1p_2) &= \frac{1}{\varphi(p_1p_2)^3} \bigg(-\mathcal{R}_0^3+p_2\mathcal{R}_1^3\sum_{\substack{r_{12},r_{22},r_{32}\\r_{12}+r_{22}+r_{32}\equiv m\bmod p_2}} 1\bigg).
\end{align}
Lastly, we have
\begin{align*}
    g(p_1p_2, p_1p_2&) 
    =\frac{1}{p_1^3p_2^3} \sum_{\substack{a=1\\(a, p_1p_2)=1}}^{p_1p_2-1}e\left(\frac{-am}{p_1p_2}\right) \underset{r_{11},r_{21},r_{31}\in \mathcal{R}_1}{\sum\sum\sum}\hspace{0.1cm}\underset{r_{12},r_{22},r_{32}\in \mathcal{R}_2}{\sum\sum\sum} \hspace{0.1cm} \underset{0 \leqslant \ell_{11}, \ell_{21}, \ell_{31} \leqslant p_1-1 }{\sum \sum \sum} 
 \hspace{0.1cm} \underset{0 \leqslant \ell_{12}, \ell_{22}, \ell_{32} \leqslant p_2-1 }{\sum \sum \sum} \\
    &\quad\times e \left (\frac{-\ell_{11} r_{11}}{p_1p_2} \right ) e\left (\frac{-\ell_{12} r_{12}}{p_1p_2} \right )e \left (\frac{-\ell_{21} r_{21}}{p_1p_2} \right ) e \left (\frac{-\ell_{22} r_{22}}{p_1p_2} \right )e \left (\frac{-\ell_{31} r_{31}}{p_1p_2} \right ) e \left (\frac{-\ell_{32} r_{32}}{p_1p_2} \right )\\
    &\quad \times  
    \frac{\mu}{\varphi}\bigg(\frac{p_1p_2}{(p_1p_2,a+\ell_{11}u_1+\ell_{12}u_2)}\bigg)\frac{\mu}{\varphi}\bigg(\frac{p_1p_2}{(p_1p_2,a+\ell_{21}u_1+\ell_{22}u_2)}\bigg)\\
    &\quad\times\frac{\mu}{\varphi}\bigg(\frac{p_1p_2}{(p_1p_2,a+\ell_{31}u_1+\ell_{32}u_2)}\bigg).
\end{align*}
Same as before, if $p_j\mid (a+\ell_i)$, then $p_j\equiv -a\bmod (p_1p_2/p_j)$. Again, only one such $\ell_{ij}$ can satisfy this condition, and so for each $\ell_{ij}$ we divide into two cases, and in total there are 64 cases. Using similar steps as in the previous cases, we obtain
\begin{align}\label{gp1p2 final form 4}
    g(p_1p_2, p_1p_2) &=\frac{\mathcal{C}(2)}{p_1^3p_2^3} \underset{r_{11},r_{21},r_{31}\in \mathcal{R}_1}{\sum\sum\sum}\hspace{0.1cm}\underset{r_{12},r_{22},r_{32}\in \mathcal{R}_2}{\sum\sum\sum}\sum_{a_1=1}^{p_1-1}\sum_{a_2=1}^{p_2-1}e\left(\frac{a_1(r_{11}+r_{21}+r_{31}-m)}{p_1p_2}\right)  \notag \\
    &\quad \times e\left(\frac{a_2(r_{12}+r_{22}+r_{32}-m)}{p_1p_2}\right) \notag \\
    &=\frac{1}{\varphi(p_1p_2)^3} \bigg(- \mathcal{R}_1^3+p_1\sum_{\substack{r_{11},r_{21},r_{31}\\r_{11}+r_{21}+r_{31}\equiv m\bmod p_1}}1 \bigg)\bigg(- \mathcal{R}_2^3+p_2\sum_{\substack{r_{12},r_{22},r_{32}\\r_{12}+r_{22}+r_{32}\equiv m\bmod p_2}}1 \bigg).
\end{align}
Putting together \eqref{gp1p2 final form 1}, \eqref{gp1p2 final form 2}, \eqref{gp1p2 final form 3} and \eqref{gp1p2 final form 4} in \eqref{Breaking G}, we get
\begin{align*}
G(p_1p_2)&=\frac{1}{\varphi(p_1p_2)^3}\bigg(\mathcal{R}_0^3-\mathcal{R}_0^3+p_1\mathcal{R}_2^3\sum_{\substack{r_{11},r_{21},r_{31}\\r_{11}+r_{21}+r_{31}\equiv m\bmod p_1}} 1-\mathcal{R}_0^3 +p_2\mathcal{R}_1^3\sum_{\substack{r_{12},r_{22},r_{32}\\r_{12}+r_{22}+r_{32}\equiv m\bmod p_2}} 1 \\
    &\quad +\mathcal{R}_0^3-p_2\mathcal{R}_1^3\sum_{\substack{r_{12},r_{22},r_{32}\\r_{12}+r_{22}+r_{32}\equiv m\bmod p_2}}1-\mathcal{R}_2^3p_1 \sum_{\substack{r_{11},r_{21},r_{31}\\r_{11}+r_{21}+r_{31}\equiv m\bmod p_1}}1\\
    &\quad +p_1p_2\sum_{\substack{r_{11},r_{21},r_{31}\\r_{11}+r_{21}+r_{31}\equiv m\bmod p_1}}\sum_{\substack{r_{12},r_{22},r_{32}\\r_{12}+r_{22}+r_{32}\equiv m\bmod p_2}}1\bigg)=G(p_1)G(p_2).
\end{align*}
This shows that the desired claim holds for $k=2$.
\par
Now assume the desired result is true for some $j\geqslant 2$. Define
\[
\mathcal{T}(j): = \text{the power set of $\{1,2,\cdots,j\}$}.
\]
Let $S$ be an element in $\mathcal{T}(j)$, and define $\bar{S}$ to be the complement of $S$. Then we can write
\begin{align*}
    g(\prod_{s\in S}p_s&,\hspace{0.1cm}  p_1p_2\cdots p_j) \\
    &= \frac{1}{(p_1p_2\cdots p_j)^3}\sum_{\substack{1\leqslant a\leqslant \prod_{s\in S}p_s\\(a,\prod_{s\in S}p_s)=1}}e\left(-\frac{am}{\prod_{s\in S}p_s}\right)\underset{r_{11},r_{21},r_{31}\in \mathcal{R}_1}{\sum\sum\sum}\hspace{0.1cm}\cdots \underset{r_{1j},r_{2j},r_{3j}\in \mathcal{R}_j}{\sum\sum\sum}\hspace{0.1cm}\underset{\ell_{11},\ell_{21},\ell_{31}\in \Z_{p_1}}{\sum\sum\sum}\hspace{0.1cm}\\
    &\quad \times\cdots\underset{\ell_{1j},\ell_{2j},\ell_{3j}\in \Z_{p_j}}{\sum\sum\sum}\hspace{0.1cm}e\left(\frac{-(\ell_{11}r_{11}+\ell_{21}r_{21}+\ell_{31}r_{31})}{p_1p_2\cdots p_j}\right)\cdots e\left(\frac{-(\ell_{1j}r_{1j}+\ell_{2j}r_{2j}+\ell_{3j}r_{3j})}{p_1p_2\cdots p_j}\right)\\
    &\quad \times\frac{\mu}{\varphi}\bigg(\frac{p_1p_2\cdots p_j}{(p_1p_2\cdots p_j,a\prod_{\bar{s}\in \bar{S}}p_{\bar{s}}+\ell_{11}u_1+\ell_{12}u_2+\cdots+\ell_{1j}u_j)}\bigg)\\
    &\quad \times  \frac{\mu}{\varphi}\bigg(\frac{p_1p_2\cdots p_j}{(p_1p_2\cdots p_j, a\prod_{\bar{s}\in \bar{S}}p_{\bar{s}}+\ell_{21}u_1+\ell_{22}u_2+\cdots+\ell_{2j}u_j)}\bigg)\\
    &\quad \times  \frac{\mu}{\varphi} \bigg(\frac{p_1p_2\cdots p_j}{(p_1p_2\cdots p_j,a\prod_{\bar{s}\in \bar{S}}p_{\bar{s}}+\ell_{31}u_1+\ell_{32}u_2+\cdots+\ell_{3j}u_j)}\bigg).
\end{align*}
Using the same reasoning as in the case for $k=2$, if $p_b\mid (a\prod_{\bar{s}\in \bar{S}}p_{\bar{s}}+\ell_{i1}u_1+\ell_{i2}u_2+\cdots+\ell_{ij}u_j)$ for $b\in S$, then $\ell_{ib}\equiv -a\prod_{\bar{s}\in \bar{S}}p_{\bar{s}}\bmod p_b$ . While for $b\in\bar{S}$, $p_b\mid \ell_{ib}$. Since each $\ell_{ib}$ run through the entire residue class modulo $p_b$, only one such $\ell_{ib}$ could satisfy the case. Therefore there are in total $(2^j)^3 = 8^j$ cases and we obtain
\begin{align}\label{General case 1}
    g(\prod_{s\in S}p_s, \hspace{0.1cm} p_1p_2\cdots p_j)&= \frac{\mathcal{C}(j)}{(p_1p_2\cdots p_k)^3}\underset{r_{11},r_{21},r_{31}\in \mathcal{R}_1}{\sum\sum\sum}\hspace{0.1cm}\cdots \underset{r_{1j},r_{2j},r_{3j}\in \mathcal{R}_j}{\sum\sum\sum}\hspace{0.1cm}\sum_{\substack{1\leqslant a\leqslant \prod_{s\in S}p_s\\(a,\prod_{s\in S}p_s)=1}}e\left(-\frac{am}{\prod_{s\in S}p_s}\right)\notag \\
    &\quad \times \prod_{s\in S}e\left(\frac{(ar_{1s}+ar_{2s}+ar_{3s})}{\prod_{s\in S}p_s}\right).
\end{align}
Using the Chinese Remainder theorem and Lemma \ref{lemma for c(k)}, the right hand side of \eqref{General case 1} is equal to
\begin{align}
&\frac{1}{\varphi(p_1p_2\cdots p_j)^3}\underset{r_{11},r_{21},r_{31}\in \mathcal{R}_1}{\sum\sum\sum}\hspace{0.1cm}\cdots \underset{r_{1j},r_{2j},r_{3j}\in \mathcal{R}_j}{\sum\sum\sum}\hspace{0.1cm}\prod_{s\in S} \left(\sum_{a_s=1}^{p_s-1}e\left(\frac{a_s(r_{1s}+r_{2s}+r_{3s}-m)}{\prod_{s\in S}p_s}\right)\right)\notag\\
    &=\frac{\prod_{\bar{s}\in\bar{S}}(\mathcal{R}_{\bar{s}})^3}{\varphi(p_1p_2\cdots p_j)^3} \prod_{s\in S} \bigg(-\mathcal{R}_s^3+p_s\sum_{\substack{r_{1s},r_{2s},r_{3s}\\r_{1s}+r_{2s}+r_{3s}\equiv m\bmod p_s}}1\bigg).\label{decomposition of g(u,q0)}
\end{align}
Now we are ready to prove the statement for $k=j+1$. Using \eqref{decomposition of g(u,q0)}, we have
\begin{align*}
    G(p_1p_2\cdots p_{j+1})=\frac{1}{\varphi(p_1p_2\cdots p_{j+1})^3}\sum_{S\in \mathcal{T}(j+1)} \prod_{\bar{s}\in \bar{S}}\mathcal{R}_{\bar{s}}^3 \prod_{s\in S} \bigg(-\mathcal{R}_s^3+p_s\sum_{\substack{r_{1s},r_{2s},r_{3j}\\r_{1s}+r_{2s}+r_{3s}\equiv m\bmod p_s}}1\bigg).
\end{align*}
Notice that for any $S\in\mathcal{T}(j+1)$, either $S= S'\cup\{j+1\}$ or $\bar{S} = \bar{S'}\cup\{j+1\}$ for some $S'\in \mathcal{T}(j)$. Using this fact together with the inductive hypothesis, we have
\begin{align*}
    G(&p_1p_2\cdots p_{j+1})\\
    &=\frac{1}{\varphi(p_{j+1})^3}\cdot\frac{1}{\varphi(p_1p_2\cdots p_j)^3}\mathcal{R}_{j+1}^3\sum_{S'\in \mathcal{T}(j)}  \prod_{\bar{s}\in \bar{S}'}\mathcal{R}_{\bar{s}}^3 \prod_{s\in S'} \bigg(-\mathcal{R}_s^3+p_s\sum_{\substack{r_{1s},r_{2s},r_{3j}\\r_{1s}+r_{2s}+r_{3s}\equiv m\bmod p_s}}1\bigg)\\
    &\quad +\bigg(-\mathcal{R}_{j+1}^3+p_{j+1}\sum_{\substack{r_{1\{j+1\}},r_{2\{j+1\}},r_{3\{j+1\}}\\r_{1\{j+1\}}+r_{2\{j+1\}}+r_{3\{j+1\}}\equiv m\bmod p_{j+1}}}1\bigg)\sum_{S'\in \mathcal{T}(j)}  \prod_{\bar{s}\in \bar{S'}}\mathcal{R}_{\bar{s}}^3\\
    &\quad\times\prod_{s\in S'} \bigg(-\mathcal{R}_s^3+p_s\sum_{\substack{r_{1s},r_{2s},r_{3s}\\r_{1s}+r_{2s}+r_{3s}\equiv m\bmod p_s}}1\bigg)\\
    &=\frac{1}{\varphi(p_{j+1})^3}\bigg(\mathcal{R}_{j+1}^3+\bigg(-\mathcal{R}_{j+1}^3+p_{j+1}\sum_{\substack{r_{1\{j+1\}},r_{2\{j+1\}},r_{3\{j+1\}}\\r_{1\{j+1\}}+r_{2\{j+1\}}+r_{3\{j+1\}}\equiv m\bmod p_{j+1}}}1\bigg)\bigg)\prod_{b=1}^j G(j)\\
    &=\prod_{b=1}^{j+1} G(p_b),
\end{align*}
where the last equality holds if we can show that for any prime $p$, 
\begin{align}\label{Final Claim}
G(p)= \frac{p}{\varphi(p)^3}\sum_{\substack{r_{11},r_{21},r_{31}\\m\equiv r_{11}+r_{21}+r_{31}\bmod p}} 1.
\end{align}
To prove \eqref{Final Claim}, we begin by writing $G(p) = g(1,p)+g(p,p)$. Then by definition of $g$ in \eqref{g(w,q_0)}, we have 
\begin{align*}
g(1,p)  &=\frac{1}{p^3}\underset{r_{11},r_{21},r_{31}\in \mathcal{R}_1}{\sum\sum\sum}\hspace{0.1cm}\underset{\ell_{11},\ell_{21},\ell_{31}\in \Z_{p}}{\sum\sum\sum}\hspace{0.1cm}e \left (\frac{-\ell_{11} r_{11}}{p} \right )e \left (\frac{-\ell_{21} r_{21}}{p} \right )e \left (\frac{-\ell_{31} r_{31}}{p} \right )\\
     &\quad \times\frac{\mu}{\varphi}\bigg(\frac{p}{(p,p+\ell_{11})}\bigg)\frac{\mu}{\varphi}\bigg(\frac{p}{(p,p+\ell_{21})}\bigg)\frac{\mu}{\varphi}\bigg(\frac{p}{(p,p+\ell_{31})}\bigg).
\end{align*}
Similar to the case for $k=2$ above, after discussing congruence relations on $\ell_{i1}$ for $i=1,2,3$, we get 
\begin{align}\label{Separate Case 1}
g(1,p) = \frac{1}{\varphi(p)^3}\mathcal{R}_0^3.
\end{align}
Following in the same manner, we write 
\begin{align*}
g(p,p)
    &=\frac{1}{p^3} \sum_{a=1}^{p-1} e\left( -\frac{am}{p}\right) \underset{r_{11},r_{21},r_{31}\in \mathcal{R}_1}{\sum\sum\sum}\hspace{0.1cm}\underset{\ell_{11},\ell_{21},\ell_{31}\in \Z_{p}}{\sum\sum\sum}\hspace{0.1cm}e \left (\frac{-\ell_{11} r_{11}}{p} \right )e \left (\frac{-\ell_{21} r_{21}}{p} \right )  e \left (\frac{-\ell_{31} r_{31}}{p} \right )\\
    &\quad \times \frac{\mu}{\phi}\bigg(\frac{p}{(p,a+\ell_{11})}\bigg)\frac{\mu}{\phi}\bigg(\frac{p}{(p,a+\ell_{21})}\bigg)\frac{\mu}{\phi}\bigg(\frac{p}{(p,a+\ell_{31})}\bigg).
\end{align*}
Note that $a+\ell_{i1}$ is at most $2p-2$, so in order to have $p\mid a+\ell_{i1}$, we need $a+\ell_{i1}=p$. Therefore we again break each $\ell_{i1}$ into two cases accordingly, and there are in total $8$ cases. We then obtain
\begin{align}\label{Separate Case 2}
g(p,p)&=\frac{1}{\varphi(p)^3}\underset{r_{11},r_{21},r_{31}\in \mathcal{R}_1}{\sum\sum\sum}\hspace{0.1cm}\sum_{a=1}^{p-1} e \left (\frac{a(r_{11}+r_{21}+r_{31}-m)}{p} \right ) \notag \\
&= \bigg(\frac{p}{\varphi(p)^3}\sum_{\substack{r_{11},r_{21},r_{31}\\m\equiv r_{11}+r_{21}+r_{31}\bmod p}} 1\bigg) -\frac{\mathcal{R}_0^3}{\varphi(p)^3}.
\end{align}
Thus combining \eqref{Separate Case 1} and \eqref{Separate Case 2}, we arrive at 
\begin{align}
  G(p) = g(1,p)+g(p,p) = \frac{p}{\varphi(p)^3}\sum_{\substack{r_{11},r_{21},r_{31}\\m\equiv r_{11}+r_{21}+r_{31}\bmod p}} 1. \label{G(p) formula}  
\end{align}
This is exactly what we were left with to show. Therefore by induction, the proof follows.
\end{proof}
\par
\subsection{A Lower Bound for $ \mathfrak{S}(m,q_0, \mathcal{R}_0)$} By Lemma \ref{breakdown og G(q0)}, we see that 
\begin{align}
    G(q_0) = G(p_1p_2\cdots p_k)=\prod_{j=1}^k G(p_j).\label{G(q0) = product of G(p)}
\end{align}
To prove that $G(q_0)>0$, it suffices to show that $G(p)>0$ for any prime $p$. By Lemma \ref{sum of three numbers modulo pj}, we have
\[
\sum_{\substack{r_{11},r_{21},r_{31}\\m\equiv r_{11}+r_{21}+r_{31}\bmod p}} 1 \geqslant 1.
\]
Therefore from the expression of $G(p)$ in \eqref{G(p) formula}, we obtain 
\begin{align}
    G(p)>\frac{p}{\varphi(p)^3}>0.\label{G(p)}
\end{align}
Finally substituting \eqref{G(p)} into \eqref{G(q0) = product of G(p)}, we conclude that
\begin{align}\label{G Nonzero}
    G(q_0) \geqslant \frac{p_1p_2\cdots p_k}{\varphi(p_1p_2\cdots p_k)^3}>0.
\end{align} 
Since $G(q_0)$ is nonzero, we conclude that the arithmetic factor $\mathfrak{S}(m,q_0, \mathcal{R}_0)$ is nonzero for odd $m$, and
\begin{align}\label{Sigma Nonzero}
\mathfrak{S}(m,q_0, \mathcal{R}_0)\geqslant\frac{p_1p_2\cdots p_k}{\varphi(p_1p_2\cdots p_k)^2}\prod_{\substack{p\nmid q_0\\p|m}}\left(1-\frac{1}{(p-1)^2}\right)\prod_{\substack{p\nmid q_0\\p\nmid m}}\left(1+\frac{1}{(p-1)^3}\right).
\end{align}
\section{Completion of the Proof of Main Theorem}\label{sec: Completion}
In what follows, we provide the final arguments for the proof of Lemma \ref{Main Lemma} and thereby complete the proof of Theorem \ref{Main Theorem}. Recall that in \eqref{Major Arcs Error Term Step 0}, we have
\begin{align}
    \int_{\mathfrak{M}}& S(\alpha)^3 e(-m \alpha) \, d   \alpha = \mathcal{H}(m,u) \mathfrak{S}(m,q_0, \mathcal{R}_0) + O(u^2(\log u)^{12A-B+\varepsilon}),\notag
\end{align}
for any $\varepsilon>0$. Note that for $m\in [4u,5u]$, $\frac{1}{2}u^2\leqslant \mathcal{H}(m,u)\leqslant \frac{3}{4}u^2$. Define 
\begin{align}\label{Sigma Prime Def}
\mathfrak{S}'(m,q_0, \mathcal{R}_0):=\mathfrak{S}(m,q_0, \mathcal{R}_0) \frac{\mathcal{H}(m,u)}{u^2}.
\end{align}
Then we have 
\begin{align*}
     \int_{\mathfrak{M}}& S(\alpha)^3 e(-m \alpha) \, d   \alpha = u^2\mathfrak{S}'(m,q_0, \mathcal{R}_0) + O(u^2(\log u)^{12A-B+\varepsilon}),
\end{align*}
where 
\begin{align}\label{Lower Bound for Sigma Prime}
\mathfrak{S}'(m,q_0, \mathcal{R}_0)\geqslant \frac{p_1p_2\cdots p_k}{2\varphi(p_1p_2\cdots p_k)^2}\prod_{\substack{p\nmid q_0\\p|m}}\left(1-\frac{1}{(p-1)^2}\right)\prod_{\substack{p\nmid q_0\\p\nmid m}}\left(1+\frac{1}{(p-1)^3}\right)>0.
\end{align}
From \eqref{Minor Arcs Prefinal Step 3}, we obtain that
\begin{align} 
\int_{\mathfrak{m}} (S(\alpha))^3 e(-\alpha m) \, d   \alpha \ll \frac{u^2}{(\log u)^{\frac{B}{2}-5A-5-2\varepsilon}}.
\end{align}
Choosing $B=\max\{13A+\varepsilon, 10A+10+4\varepsilon\}$, we see that $\mathfrak{R}(m)$, which is defined to be the number of representations of writing an integer $m\in [4u,5u]$ as a sum of three prime numbers from the set $\mathbb{P}_u$ weighted by the Van-Mangoldt function in \eqref{Representation of m}, is as follows: 
\begin{align}\label{Asymptotic Formula}
    \mathfrak{R}(m)&=\int_{\mathfrak{m}} (S(\alpha))^3 e(-\alpha m) \, d   \alpha +\int_{\mathfrak{M}} S(\alpha)^3 e(-m \alpha) \, d   \alpha \notag \\
    &=u^2\mathfrak{S}'(m,q_0, \mathcal{R}_0) + O\bigg(\frac{u^2}{(\log u)^A}\bigg), 
\end{align}
where $\mathfrak{S}'(m,q_0, \mathcal{R}_0)>0$. This proves part (b) of Lemma \ref{Main Lemma}. Moreover, recall that our construction of the set $\mathbb{P}_u$ with such $\mathcal{R}_0$ already satisfies part (a) of the lemma, so the proof of Lemma \ref{Main Lemma} is complete. Hence Theorem \ref{Main Theorem} follows.
\bibliographystyle{plain}

\begin{thebibliography}{99}
\bibitem{bordignon2022}
M. Bordignon, D. R. Johnston, V. Starichkova, \textit{An Explicit Version of Chen's Theorem}, (2024), {arXiv:2207.09452v4}.  


\bibitem{cai2002}
Y. C. Cai, \textit{Chen's Theorem with Small Primes}, Acta Mathematica Sinica.18 (2002), (3): 597–604.

\bibitem{chen1966}
J. R. Chen, \textit{On the representation of a large even integer as the sum of a prime and the product of at most two primes}, Kexue Tongbao. 11 (1966), 385–386.

\bibitem{chen1973}
 J. R. Chen, \textit{On the representation of a larger even integer as the sum of a prime and the product of at most two primes}, Sci. Sinica. 16 (1973), 157–176.

\bibitem{chen1989}
J. R. Chen, J. M. Liu, \textit{The exceptional set of Goldbach numbers III}, Chinese Quart. J. Math. 4 (1989), 1–15.

\bibitem{chen_2_1989}
J. R. Chen, J. M. Liu, \textit{On the least prime in an arithmetical
progression}, Sci. China Ser. A 32 (1989), (6): 654–673
and (7): 792–807.

\bibitem{cudakov1938}
N. G. Cudakov, \textit{On the density of the set of even numbers which are
not representable as a sum of two primes}, Izv. Akad. Nauk SSSR 2
(1938), 25–40.



\bibitem{D2000}
H. Davenport, \emph{Multiplicative Number Theory,} Springer (2000), Graduate Texts in Mathematics v.74, 3rd edition.

\bibitem{estermann1938}
T. Estermann, \textit{On Goldbach’s problem: Proof that almost all even
positive integers are sums of two primes}, Proc. London Math. Soc.(1938), (2): 307–314.


\bibitem{gallagher1975}
 P. X. Gallagher, \textit{Primes and powers of 2}, Invent. Math. 29 (1975), 125–142.

\bibitem{goldston1992} D. A. Goldston, \textit{On Hardy and Littlewood’s contribution to the Goldbach conjecture}, Proceedings of the Amalfi Conference on Analytic Number Theory (Maiori, 1989), 115–155, Univ. Salerno, Salerno, 1992.

\bibitem{hardy1922}
 G. H. Hardy, J. E. Littlewood, \textit{Some problems of `Partitio numerorum'; III: On the expression of a number as a sum of primes},  Acta Math. (44) (1923), no. 1, 1--70. 

\bibitem{hardy1924}
 G. H. Hardy, J. E. Littlewood, \textit{Some problems of ‘Partitito Numerorum’, V: A further contribution to the study of Goldbach’s problem}, Proc. London Math. Soc.(1924), (2): 46–56.
 
 \bibitem{heathbrown1992}
D. R. Heath-Brown, \textit{Zero-free regions for Dirichlet L-functions, and
the least prime in an arithmetic progression}, Proc. London Math.
Soc. (3) 64 (1992), no. 2, 265–338.


\bibitem{HeathBrown2002}
D. R. Heath-Brown, J. C. Puchta, \textit{Integers represented as a sum of primes and powers of two}, Asian J. Math. 6 (2002), 535–565. 

\bibitem{helfgott2013}
H. A. Helfgott, \textit{The ternary Goldbach problem}, to appear in Ann. of Math. Studies.


\bibitem{Khalfalah2006}
A. Khalfalah, J. Pintz, \textit{On the representation of Goldbach numbers by a bounded number of powers of two}, in: Elementare und analytische Zahlentheorie, Schr. Wiss. Ges. Johann Wolfgang Goethe Univ. Frankfurt am Main, vol. 20 (2006), 129–142.

\bibitem{K2019}
D. Koukoulopoulos, \emph{The distribution of prime numbers,} Graduate Studies in Mathematics, vol. 203, American Mathematical Society (2019), Providence, RI.



\bibitem{li2000_powers}
H. Z. Li, \textit{The number of powers of 2 in a representation of large even integers by sums of such powers and of two primes}, Acta Arith. 92 (2000), 229–237.


\bibitem{li2000}
H. Z. Li, \textit{The exceptional set of Goldbach numbers I}, Quart J.
Math. Oxford Ser. (2) 50 (2000), no. 200, 471–482.

\bibitem{li_2_2000}
H. Z. Li, \textit{The exceptional set of Goldbach numbers II}, Acta Arith.
92 (2000), no. 1, 71–88.

\bibitem{Linnik1951}
Y. Linnik, \textit{Prime numbers and powers of two}, Trudy Mat. Inst. Steklov 38 (1951), 152–169.


\bibitem{lucai1998} M. Lu and Y. Cai, \textit{Chen’s theorem in short intervals}, Chinese Sci. Bull., 43(16):1401–1403, 1998. 

\bibitem{lucai1999} M. Lu and Y. Cai, \textit{Chen’s theorem in arithmetical progressions}, Sci. China Ser. A, 42(6):561–569, 1999.

\bibitem{lu2010}
W. C. Lu, \textit{Exceptional set of Goldbach number}, J. Number Theory 130 (2010), no. 10, 2359-2392.



\bibitem{montgomery1975}
H. L. Montgomery, R. C. Vaughan, \textit{The exceptional set in Goldbach’s
problem}, Collection of articles in memory of Juri˘i Vladimiroviˇc Linnik, Acta Arith. 27 (1975), 353-370.



\bibitem{pan1962}
C. T. Pan, \textit{On the representation of an even number as the sum of a prime and of an almost prime}, Acta Mathematica Sinica (1962), 12, 95-106.


\bibitem{Pintz2003}
J. Pintz, I. Z. Ruzsa, \textit{On Linnik's approximation to Goldbach's problem. I}, Acta Arith. 109  (2003), 169–194.

\bibitem{pintz2018}
J. Pintz, \textit{A new explicit formula in the additive theory of
primes with applications II. The exceptional set in
Goldbach’s problem}, (2018), {arXiv:1804.09084}.

\bibitem{renyi1948}
  A. A. Rényi, \textit{On the representation of an even number as the sum of a prime and an almost prime}, Izvestiya Akademii Nauk SSSR Seriya Matematicheskaya (1948), 57–78.


\bibitem{ross1975}
P. M. Ross, \textit{On Chen's theorem that each large even number has the form ($p_1+p_2$) or $(p_1+p_2p_3)$}, J. London Math. Soc. Series 2. 10 (1975), 500–506. 


\bibitem{corput1937}
J. G. van der Corput, \textit{Sur l’hypoth´ese de Goldbach pour Presque tous
les nombres pairs}, Acta Arith. 2 (1937), 266–290.



\bibitem{V1997}
R.C.Vaughan, \textit{The Hardy-Littlewood Method}, Cambridge tracts in Mathematics; 125, 2nd Ed. 1997.

\bibitem{vaughan1972}
R. C. Vaughan, \textit{On Goldbach’s problem}, Acta Arith. 22 (1972), 21-48.



\bibitem{vinogradov1937}
I. M. Vinogradov, \textit{Representation of an odd number as a sum of three prime numbers}, Doklady Akad. Nauk SSSR 15 (1937), 291–294.





\bibitem{wang1962}
Y. Wang, \textit{On the representation of large integer as a sum of a prime and an almost prime}, Sci. Sin., 11 (1962), 1033-1054.



\end{thebibliography}

\end{document}